\documentclass[12pt]{amsart}
\usepackage{amsmath}
\usepackage{amsthm}
\usepackage{amssymb}
\usepackage{tikz-cd}
\usepackage[T1]{fontenc}
\usepackage[utf8]{inputenc}
\usepackage[english]{babel}
\usepackage[colorlinks=true,linkcolor=red]{hyperref}
\usepackage[normalem]{ulem}
\usepackage{caption}

\usepackage{tikz}
\usetikzlibrary{shapes}
\usepackage[autostyle]{csquotes}
\usepackage[shortlabels]{enumitem}
\usepackage{xspace}
\usepackage{extarrows}
\usepackage{comment}

\makeatletter
\@namedef{subjclassname@2020}{%
  \textup{2020} Mathematics Subject Classification}
\makeatother

\usepackage[backend=biber, style=numeric, giveninits=true, maxbibnames=99]{biblatex}

\DeclareFieldFormat*{title}{\mkbibemph{#1}}
\DeclareFieldFormat*{citetitle}{\mkbibemph{#1}}
\DeclareFieldFormat{journaltitle}{#1}

\theoremstyle{definition}
\newtheorem{defn}{Definition}[section]
\newtheorem{example}[defn]{Example}
\newtheorem{remark}[defn]{Remark}

\theoremstyle{plain}
\newtheorem{lemma}[defn]{Lemma}
\newtheorem{prop}[defn]{Proposition}
\newtheorem{thm}[defn]{Theorem}
\newtheorem{cor}[defn]{Corollary}

\newcommand\restr[2]{{
  \left.\kern-\nulldelimiterspace 
  #1 
  \vphantom{\big|} 
  \right|_{#2} 
  }}

\newcommand{\F}{\mathcal{F}}
\newcommand{\Id}{\mathrm{Id}}
\DeclareMathOperator{\Aut}{\mathrm{Aut}}
\DeclareMathOperator{\Homeo}{\mathrm{Homeo}}

\newcommand{\G}{\mathbb{G}}
\renewcommand{\Bbb}{\mathbb}
\newcommand{\ord}{\mathrm{ord}}
\newcommand{\Fraisse}{Fra\"{i}ss\'{e}\xspace}
\newcommand{\Wazewski}{Wa\.{z}ewski\xspace}
\newcommand{\T}{\mathcal{T}}
\newcommand{\Kubis}{Kubi\'{s}\xspace}
\newcommand{\K}{\mathfrak{K}}
\newcommand{\dF}{\ddagger\mathcal{F}_P}
\newcommand{\sdF}{\sigma{\ddagger}\mathcal{F}_P}

\renewcommand{\L}{\mathcal{L}}
\DeclareMathOperator{\End}{\mathrm{End}}

\DeclareMathOperator{\osc}{\mathrm{osc}}
\DeclareMathOperator{\diam}{\mathrm{diam}}

\renewbibmacro{in:}{}

\begin{document}


\baselineskip=17pt


\title[Projective fra\"{i}ss\'{e} limits and wa\.{z}ewski dendrites]{Projective fra\"{i}ss\'{e} limits and generalized wa\.{z}ewski dendrites}

\author{Alessandro Codenotti}
\address{A. Codenotti, Institut f\"{u}r Mathematische Logik und Grundlagenforschung, Universit\"{a}t M\"{u}nster, Einsteinstrasse 62, 48149 M\"{u}nster, Germany}
\email{acodenot@uni-muenster.de}

\author{Aleksandra Kwiatkowska}
\address{A. Kwiatkowska, Institut f\"{u}r Mathematische Logik und Grundlagenforschung, Universit\"{a}t M\"{u}nster, Einsteinstrasse 62, 48149 M\"{u}nster, Germany {\bf{and}} Instytut Matematyczny, Uniwersytet Wroc{\l}awski, pl. Grunwaldzki 2/4, 50-384 Wroc{\l}aw, Poland}
\email{kwiatkoa@uni-muenster.de}

\thanks{Funded by the Deutsche Forschungsgemeinschaft (DFG, German Research Foundation) under Germany’s Excellence Strategy EXC 2044–390685587, Mathematics Münster: Dynamics– Geometry–Structure and by CRC 1442 Geometry: Deformations and Rigidity.}

\begin{abstract}
We continue the study of projective \Fraisse limits of trees initiated by Charatonik and Roe and further continued by Charatonik, Kwiatkowska, Roe and Yang. We construct many generalized \Wazewski dendrites as  topological realizations of  projective \Fraisse limits of families of finite trees with (weakly) coherent epimorphisms. Moreover we use the categorical approach to \Fraisse limits developed by \Kubis to construct all generalized \Wazewski dendrites as topological realizations of \Fraisse limits of suitable categories of finite structures. As an application we recover a homogeneity result for countable dense sets of endpoints in generalized \Wazewski dendrites.
\end{abstract}
\maketitle

\maketitle

\section{Introduction}
Projective \Fraisse limits were introduced by Irwin and Solecki in \cite{Irwin_Solecki} by dualizing the classical construction of \Fraisse limits from model theory, as a tool to study the pseudo-arc. Panagiotopolous and Solecki
generalised their construction in \cite{A_combinatorial_model} by 
allowing strict subcollections of epimorphism, instead of considering all of them. They used this approach to study the Menger curve by approximating it with finite graphs with monotone epimorphisms. Projective \Fraisse limits have since been used to study various compact spaces, for example the Lelek fan by Barto\v{s}ov\'{a} and Kwiatkowska \cite{BK}, the Poulsen simplex by \Kubis and Kwiatkowska \cite{KK}, $P$-adic pseudo-solenoids (for a set of primes $P$)
by Barto\v{s} and \Kubis \cite{BaKu},
the universal Knaster continuum by Iyer \cite{Knaster} (which has independently been constructed from  a projective \Fraisse limit in  Wickman's PhD thesis \cite{Wi}) and the generic object for a new class of compact metric spaces named \emph{fences} by Basso and Camerlo \cite{BC}. Recently Charatonik, Kwiatkowska, Roe and Yang studied projective \Fraisse limits of trees with monotone and confluent epimorphisms, which are adaptations to topological graphs of corresponding well established notions from continuum theory, see \cite{CR}. 

Dendrites are one-dimensional continua which have been widely studied in topological dynamics, \cite{DuchesneMonod1}, \cite{DuchesneMonod2}, \cite{AbNa}, \cite{MaNa}, \cite{GMtreelike_compact_spaces}. By a result of Charatonik and Dilks \cite{CD}, for any $P\subseteq\{3,4,\ldots,\omega\}$, there exists a unique dendrite
$W_P$ such that: \begin{enumerate}
    \item every ramification point of $W_P$ has order in $P$;
    \item for every $p\in P$ the set of ramification points of $W_P$ of order $p$ is arcwise dense in $W_P$.
\end{enumerate}
The spaces $W_P$ are known as the (generalized) \Wazewski dendrites, they are universal for appropriate classes of dendrites and enjoy good homogeneity properties. Because of these regularity properties they are of particular interest among dendrites. Their homeomorphism groups have been studied by Duchesne \cite{Duchesne}, while the universal minimal flows of those groups were computed by Kwiatkowska \cite{kwiatkowskaUMFofWazewski}.

 Among the spaces constructed in \cite{CR} there is the \Wazewski dendrite $W_3$ (denoted by $D_3$ in their paper), obtained from the projective \Fraisse limit of the family of  finite trees with ramification points of order 3 and all monotone epimorphisms; we build on their work to realize many generalized \Wazewski dendrites $W_P$ for $P\subseteq \{3,4,5,\ldots,\omega\}$ as the topological realizations of projective \Fraisse limits of appropriately chosen families of finite trees and epimorphisms between them. 

In particular, we obtain the following theorem. See Definition \ref{def: families Fp and Gp}, the remarks following it, and Theorems \ref{thm: limit of G_P} and \ref{thm: limit of F_P} for details.
\begin{thm}\label{thm: main thm1}
    Let $P\subseteq\{3,4,\ldots,\omega\}$ be coinfinite. Then the \Wazewski dendrite $W_P$ can be constructed as the topological realization of the projective \Fraisse limit of finite trees with monotone maps with additional properties.
\end{thm}

By moving to the more general setting of \Fraisse categories developed in \cite{K} we remove the coinfiniteness assumption from the previous result and obtain the following general statement, see Theorems \ref{thm: limit of dF is W_P} \and \ref{thm: limit of dF with countable dense set of endpoints} for details.
\begin{thm}\label{thm: main thm2}
    Let $P\subseteq\{3,4,\ldots,\omega\}$ and let $E\subseteq W_P$ be a countable dense set of endpoints. Then the pair $(W_P,E)$ can be constructed as the topological realization of the limit of a \Fraisse category $\dF$ of pairs $(T,F)$ with projection-embedding pairs of maps, where $T$ is a finite tree and $F$ is a set of endpoints in $T$.
\end{thm}

Moreover we show in Theorem \ref{thm: sequence is Fraisse if limit is WP-prespace} that any sequence in the category $\dF$ whose limit has as topological realization $W_P$, must be a \Fraisse sequence for $\dF$.
 
 A \Fraisse theoretic approach to the study of \Wazewski dendrites has already been employed successfully by Kwiatkowska in \cite{kwiatkowskaUMFofWazewski}, in which a \Fraisse structure $A_P$ with $\Aut(A_P)\simeq\Homeo(W_P)$ is constructed. While this approach is well suited to study properties of $\Homeo(W_P)$, it cannot be used to study the endpoints of $W_P$, since the injective \Fraisse limit construction produces a countable structure corresponding to the ramification points of $W_P$. In contrast, we use the uniqueness of \Fraisse sequences in the last section to recover a homogeneity result for the endpoints of $W_P$ from \cite{CD}, see Theorem \ref{thm: CDH}.

The paper is structured as follows. In Sections 2 and 3 we recall some background notions concerning projective \Fraisse limits and topological graphs. In Section 4 we introduce a new class of maps between finite trees and use it to prove Theorem \ref{thm: main thm1}. Finally in Section 5 we describe projection-embedding \Fraisse categories of finite trees that can be used to prove Theorem \ref{thm: main thm2}.

\section{Projective \Fraisse Limits} 

In this section we recall the necessary background on projective \Fraisse limits. Most definitions and results are from either \cite{Irwin_Solecki} or \cite{A_combinatorial_model}.

\begin{defn}
	Let $L$ be a first order language consisting of a set $\{R_i\}_{i\in I}$ of relation symbols, with $R_i$ of arity $n_i\in \Bbb N$, and a set $\{f_j\}_{j\in J}$ of function symbols, with $f_j$ of arity $m_j$. A \emph{topological $L$-structure} is a zero-dimensional, compact, metrizable topological space $A$  together with closed subsets $R_i^A\subseteq A^{n_i}$ for every $i\in I$ and continuous functions $f_j^A\colon A^{m_j}\to A$ for every $j\in J$. 
\end{defn}
\begin{defn}\label{def: epimorphism}
	Given two topological $L$-structures $A$ and $B$ we say that $g\colon B\to A$ is an \emph{epimorphism} if $g$ is a continuous surjection satisfying
	\begin{enumerate}
		\item For all $j\in J$, $f_j^A(g(b_1),\ldots,g(b_{m_j}))=g(f_j^B(b_1,\ldots,b_{m_j}))$ 
		\item For all $i\in I$, we have $R_i^A(a_1,\ldots,a_{n_i})$ if and only if there are $b_1\in g^{-1}(a_1),\ldots,b_{n_i}\in g^{-1}(a_{n_i})$ such that $R_i^B(b_1,\ldots,b_{n_i})$ holds.
	\end{enumerate}
\end{defn}

We can now give the definition of a projective \Fraisse family and its projective \Fraisse limit. We follow the construction from \cite{A_combinatorial_model}. The main difference compared to the original construction from \cite{Irwin_Solecki} is allowing a restricted class of epimorphisms between the structures under consideration.

\begin{defn}
	Let $\mathcal F$ be a class of finite topological $L$-structures with a fixed class of epimorphisms between them. We say that $\mathcal F$ is a \emph{projective \Fraisse family} if it satisfies the following conditions:
	\begin{enumerate}
		\item Up to isomorphism there are countably many $L$-structures in $\mathcal F$.
		\item The epimorphisms in $\mathcal F$ are closed under composition and the identity epimorphisms are in $\mathcal F$.
		\item For all $A,B\in\mathcal F$ there are $C\in\F$ and epimorphisms $f\colon C\to A$ and $g\colon C\to B$ in $\mathcal F$. In analogy with the classical \Fraisse limits this property is known as the \emph{joint projection property}.
		\item For all $A,B,C\in\mathcal F$ and epimorphisms $f\colon B\to A$ and $g\colon C\to A$ in $\mathcal F$, there exist $D\in\F$ and epimorphisms $h\colon D\to B$ and $k\colon D\to C$ in $\mathcal F$ such that $f\circ h=g\circ k$. This property is known as the \emph{projective amalgamation property}.
	\end{enumerate}
\end{defn}

Following \cite{A_combinatorial_model} we introduce $\mathcal F^\omega$ as the class of structures that can be approximated in $\mathcal F$, indeed the projective \Fraisse limit will be an element of $\mathcal F^\omega$ satisfying appropriate universality and homogeneity properties.

\begin{defn}
	Let $\mathcal F$ be a class of finite topological $L$-structures. We define $\mathcal F^\omega$ to be the class of all topological $L$-structures in $\F$ that are realized as the inverse limits of an $\omega$-indexed inverse system of topological $L$-structures with bonding epimorphisms in $\mathcal F$. More explicitly $F\in\mathcal F^\omega$ iff there is an inverse system of topological $L$-structures $\langle F_i,f^n_m\colon F_n\to F_m,i,n,m<\omega\rangle$ with $F_i,f^n_m\in\mathcal F$, such that $F=\varprojlim F_i$ (note in particular that $\mathcal F\subseteq\mathcal F^\omega$ by picking constant sequences).
	 
	Note that, given an inverse system $\langle F_i,f^n_m\colon F_n\to F_m,i,n,m<\omega\rangle$ in $\mathcal F$ as above, $F=\varprojlim F_i$ exists. Its underlying space is the inverse limit of the $F_i$ in the category of topological spaces, while for $f^n\colon F\to F_n$, the continuous maps obtained by construction of the limit of topological spaces, we have $$R_i^F(a_1,\ldots,a_{n_i})\iff \forall n R_i^{F_n}(f^n(a_1),\ldots,f^n(a_{n_i}))$$ and 
	$$f_j^F(a_1,\ldots,a_{m_j})=a\iff\forall n f_j^{F_n}(f^n(a_1),\ldots,f^n(a_{m_j}))=f^n(a).$$
	
	If $F=\varprojlim F_i\in\mathcal F^\omega$ and $E\in \mathcal F$, an epimorphism $f\colon F\to E$ is in $\mathcal F^\omega$ if and only if there exist $i\in\Bbb N$ and an epimorphism $f'\colon F_i\to E$ with $f=f'\circ f^i$. If $G$ is another object in $F^\omega$ an epimorphism $g\colon G\to F$ is in $\mathcal F^\omega$ iff $f^i\circ g$ is in $\mathcal F^\omega$ for all $i$. 
\end{defn}

\begin{thm}[\protect{\cite[Theorem 3.1]{A_combinatorial_model}}]\label{thm: projective limit exists}
	Let $\mathcal F$ be a projective \Fraisse family of topological $L$-structures. Then there exists a unique topological $L$-structure $\mathbb F\in\mathcal F^\omega$ satisfying the following properties:
	\begin{enumerate}
		\item For every $A\in\mathcal F$ there is an epimorphism $f\colon\mathbb F\to A$ in $\mathcal F^\omega$.
		\item For all $A,B\in\mathcal F^\omega$ and all epimorphisms $f\colon \mathbb F\to A$ and $g\colon B\to A$ in $\mathcal F^\omega$, there exist an epimorphism $h\colon \mathbb F\to B$ in $\mathcal F^\omega$ such that $f=g\circ h$.
	\end{enumerate}
Moreover, as pointed out in \cite[Lemma 2.2]{Irwin_Solecki} in terms of covers, we also have that $\mathbb F$ satisfies
\begin{enumerate}\setcounter{enumi}{2}
	\item For every $\varepsilon>0$ there is $A\in\mathcal F$ and an epimorphism $f\colon \mathbb F\to A$ in $\mathcal F^\omega$ such that $\mathrm{diam}(f^{-1}(a))\leq\varepsilon$ for every $a\in A$. 
\end{enumerate}
\end{thm}
\begin{proof}
	We give a sketch of the argument, the details can be found in \cite{Irwin_Solecki} and \cite{A_combinatorial_model}. The idea is to obtain $\Bbb F$ as the inverse limit of a sequence $\langle F_i \mid i\in\Bbb N\rangle$ with maps $f^n_m\colon F_n\to F_m$ of elements and epimorphisms of $\F$. We will call this sequence a \emph{\Fraisse sequence} for $\F$. Since there are only countably many finite structures in $\F$ up to isomorphism, we can build two countable lists $\langle A_i\mid i\in\Bbb N\rangle$ and $\langle e_n\colon B_n\to C_n\rangle$ so that 
	\begin{itemize}
		\item Every structure in $\F$ is isomorphic to $A_i$ for some $i$.
		\item Every epimorphism type $e\colon B\to C$ in $\F$ appears infinitely many times in $\langle e_n\mid n\in\Bbb N\rangle$.
	\end{itemize} 
    Now set $F_0=A_0$. Inductively assume that $F_n$ together with the epimorphisms $f^n_m\colon F_n\to F_m$ for all $m<n$ have been defined. By the joint projection property we can find $H\in\F$ with epimorphisms $g\colon H\to A_{n+1}$ and $f\colon H\to F_n$. Since $H$ is finite there are only finitely many epimorphism types $s_i\colon H\to C_{n+1}$, $i=1,\ldots,k$. Using the projective amalgamation property, we can find $H_1'$ and epimorphisms $h_1\colon H_1'\to H$, $c_1\colon H_1'\to B_{n+1}$ such that $s_1\circ h_1= e_{n+1}\circ c_1$. Iterating this construction, for all $n\leq k$, we can find $H_n'$ with epimorphisms $h_n\colon H_n'\to H_{n-1}'$ and $c_n\colon H_n'\to B_{n+1}$ such that $s_n\circ h_1\circ\ldots\circ h_n=e_{n+1}\circ c_n$. Setting now $H'=H_k'\in\F$, $h=h_1\circ\ldots h_k\colon H'\to H$, $d_i=c_i\circ h_{i+1}\circ\ldots\circ h_k\colon H'\to B_{n+1}$, $i=1,\ldots,k$, we have that the following diagram commutes, that is $s_i\circ h=e_{n+1}\circ d_i$ for all $i$.
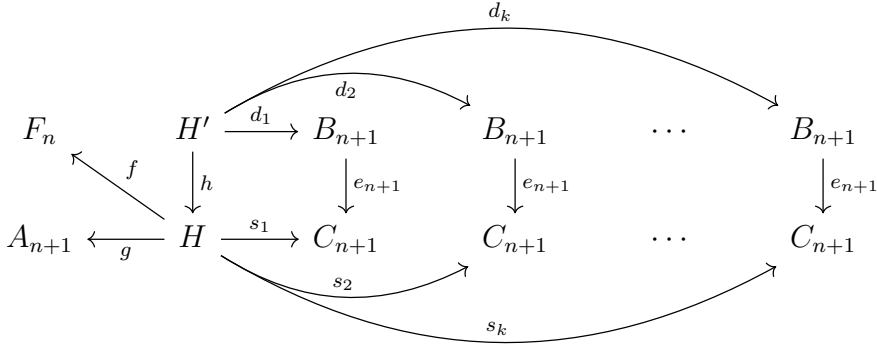
    \begin{center}
    \begin{tikzcd}
    F_n     & H' \arrow[d, "h"] \arrow[r, "d_1"] \arrow[rr, "d_2"', bend left] \arrow[rrrr, "d_k", bend left]                    & B_{n+1} \arrow[d, "e_{n+1}"] & B_{n+1} \arrow[d, "e_{n+1}"] & \cdots & B_{n+1} \arrow[d, "e_{n+1}"] \\
    A_{n+1} & H \arrow[r, "s_1"] \arrow[rr, "s_2", bend right] \arrow[rrrr, "s_k", bend right] \arrow[l, "g"] \arrow[lu, "f"'] & C_{n+1}                      & C_{n+1}                      & \cdots & C_{n+1}                     
    \end{tikzcd}
    \end{center}
    Once this is done set $F_{n+1}=H'$, and set $f^{n+1}_m=f^n_m\circ f\circ h$ for all $m<n$. It is easy to verify that $\Bbb F=\varprojlim F_i$ satisfies the first two properties of a projective \Fraisse limit of $\mathcal F$. The third one follows from the definition of the product metric on $\prod F_i$.
\end{proof}

From the proof of Theorem \ref{thm: projective limit exists} we extract the following important notion.

\begin{defn}
	Let $\F$ be a projective \Fraisse family of topological $L$-structures. A sequence $\langle F_n\mid n\in\Bbb N\rangle$ with epimorphisms $f^m_n\colon F_m\to F_n$ is called a \emph{\Fraisse sequence} for $\F$ if the following conditions hold:
	\begin{itemize}
		\item for every $A\in\F$ there exists $n\in\Bbb N$ such that there is an epimorphism $f\colon F_n\to A$ in $\F$,
		\item whenever $f\colon A\to F_m$ is an epimorphism in $\F$, there exists $n\in\Bbb N$ such that there is an epimorphism $g\colon F_n\to A$ in $\F$ with $f\circ g=f^n_m$. 
	\end{itemize}
\end{defn}

\section{Topological Graphs and Monotone Epimorphisms}
In this section we reproduce for convenience definitions and results about topological graphs that will be used in Section 4 and Section 5. They are mostly taken from \cite{CR}. 

Theorems \ref{thm: transitive edges}, \ref{thm: monotone maps preserve endpoints}, \ref{thm: limit is arcwise connected}, 
\ref{thm: limit is unicoherent}, \ref{thm: limit is dendrite} first appeared in the preprint of Charatonik-Roe \cite{CR-old}, but since that preprint is unpublished, with the permission of the authors, we include the proofs for completeness.
Those results will appear in \cite{CR}.

\subsection{Topological Graphs and epimorphisms}
\begin{defn}
A \emph{graph} $G$ is a pair $(V(G),E(G))$ consisting of a set of \emph{vertices} $V(G)$ and a set of \emph{edges} $E(G)\subseteq V(G)^2$ so that 
\begin{itemize}
	\item For all $a,b\in V(G)$, $\langle a,b\rangle\in E(G)\implies\langle b,a\rangle\in E(G)$.
	\item For all $a\in V(G)$, $\langle a,a\rangle\in E(G)$.
\end{itemize}
A finite graph $T$ is a {\it tree } if there are no pairwise distinct vertices $a_1,\ldots, a_n\in T$, $n\geq 3$, such that $\langle a_i, a_{i+1}\rangle \in E(T)$ for $i=1,2,\ldots, n-1$, and $\langle a_n, a_1\rangle \in E(T)$. 
We will use interchangeably $a\in V(A)$ and $a\in A$ and we call the edges $\langle a,b\rangle\in E(A)$ with $a\neq b$ \emph{nontrivial}.  
A \emph{topological graph} is a graph $G$ equipped with a compact, Hausdorff, zero-dimensional, metrizable topology on $V(G)$ so that $E(G)$ is a closed subspace of $V(G)^2$. Every topological graph is a topological $L$-structure for the language $L=\{R\}$ consisting of a single binary relation in an obvious way.  We will keep using the notation $\langle a,b\rangle\in E(G)$ rather than $aRb$ even though we will often think about topological graphs as topological $L$-structures. An epimorphism $f\colon A\to B$ of topological graphs is an epimorphism in the sense of Definition \ref{def: epimorphism}, so in particular it maps edges to edges.

Let us point out that in this context requiring topological graphs to be zero-dimensional is a natural property, since we are interested in inverse limits of finite discrete structures. We warn the reader that in the literature there are different definitions of  topological graphs, 
where edges are embedded as arcs into the topological structure.

If $A$ is a topological graph and $B\subseteq A$, $A\setminus B$ denotes the topological graph with $V(A\setminus B)=V(A)\setminus V(B)$ and for all $a_1,a_2\in V(A\setminus B)$, $$\langle a_1,a_2\rangle\in E(A\setminus B)\iff \langle a_1,a_2\rangle\in E(A).$$

A topological graph $G$ is said to have a \emph{transitive set of edges} if $$\langle a,b\rangle,\langle b,c\rangle\in E(G)\implies\langle a,c\rangle\in E(G).$$ It is called a \emph{prespace} if the edge relation $aRb\iff \langle a,b\rangle\in E(G)$ is an equivalence relation, equivalently a prespace is a topological graph with a transitive set of edges. In that case we call the quotient topological space $|G|=V(G)/R$ the \emph{topological realization} of $G$.
\end{defn}

\begin{example}\label{example: cantor}
Let $\mathcal C\subseteq [0,1]$ be the standard middle-thirds Cantor set. Consider the topological graph $G$ with $V(G)=\mathcal C$ and for distinct $x,y\in \mathcal C$ we have $\langle x,y\rangle\in E(G)$ if and only if $x,y$ are the endpoints of one of the intervals removed from $[0,1]$ in the construction of $\mathcal C$ (see the picture below for the first few stages of the construction, with the edges represented by solid black lines). 
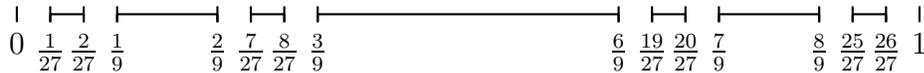
\begin{center}
	\begin{tikzpicture}[thick]
	\draw (4/3,0) -- (8/3,0) (4,0) -- (8,0) (28/3,0) -- (32/3,0);
	\draw (12/27,0) -- (24/27,0) (84/27,0) -- (96/27,0) 
	(228/27,0) -- (240/27,0) (300/27,0) -- (312/27,0);
	\foreach \x in  {1,2,3,6,7,8}{
	\draw[shift={(12/9*\x,0)},color=black] (0pt,3pt) -- (0pt,-3pt);
	\draw[shift={(12/9*\x,0)},color=black] (0pt,0pt) -- (0pt,-3pt) node[below] {$\frac{\x}{9}$};
    }

	\draw[shift={(0,0)},color=black] (0pt,3pt) -- (0pt,-3pt) node[below] {$0$};
	\draw[shift={(12,0)},color=black] (0pt,3pt) -- (0pt,-3pt) node[below] {$1$};	

    \foreach \x in {1,2,7,8,19,20,25,26}{
    \draw[shift={(12/27*\x,0)},color=black] (0pt,3pt) -- (0pt,-3pt);
    \draw[shift={(12/27*\x,0)},color=black] (0pt,0pt) -- (0pt,-3pt) node[below] {$\frac{\x}{27}$};
    }
	\end{tikzpicture}
\end{center}

\noindent Then $G$ is a prespace, in particular $|G|\cong [0,1]$.
\end{example}

\begin{defn}\label{defn: arcs and more}
	A topological graph $G$ is called \emph{disconnected} if it is possible to partition $G$ into two nonempty disjoint closed sets $A,B$ such that whenever $a\in A$, $b\in B$, we have $\langle a,b\rangle\not\in E(G).$
	A topological graph is called \emph{connected} if it is not disconnected. A topological graph $A$ is called an \emph{arc} if it is connected, but for all $a\in A$, except at most two vertices, called the endpoints, $A\setminus\{a\}$ is disconnected. 
\end{defn}

\begin{example}
	The Cantor graph of Example \ref{example: cantor} is a connected  topological graph. In fact it is an arc. Note that while there is a unique topological arc, namely the interval $[0,1]$, there are many nonisomorphic topological graphs which are arcs. Indeed, there are finite, countable and uncountable topological graphs which are arcs.
\end{example}

The notion of a monotone map from continuum theory, where a continuous map $f\colon X\to Y$ between continua is called monotone if $f^{-1}(y)$ is connected for every $y\in Y$, was adapted to topological graphs by Panagiotopoulos and Solecki in \cite{A_combinatorial_model} (where those maps were called connected epimorphism instead) and studied in detail by Charatonik, Kwiatkowska, Roe, and Yang in \cite{CR}. The epimorphisms we will consider in the following sections will always be monotone so we will use various results from \cite{CR}, whose statements and proofs we include here for completeness. 

\begin{defn}
	An epimorphism $f\colon G\to H$ between topological graphs is called \emph{monotone} if $f^{-1}(h)$ is connected in $G$ for every $h\in H$.
\end{defn}

By Lemma 1.1 of \cite{A_combinatorial_model} we can equivalently require that $f^{-1}(C)$ is connected in $G$ whenever $C\subseteq H$ is closed connected in $H$. From this equivalent definition it is clear that if $f\colon G\to H$ is monotone and $H$ is connected then $G$ is too. The converse of this statement actually holds for arbitrary epimorphisms, an observation repeatedly used in both \cite{A_combinatorial_model} and \cite{CR}:
\begin{remark}\label{remark: connected image}
	Let $f\colon G\to H$ be an epimorphism between topological graphs. If $G$ is connected then $H$ is connected as well.
\end{remark}
\begin{thm}[\protect{\cite[Theorem 2.17]{CR}}] \label{thm: transitive edges}
	Let $\mathcal G$ be a projective \Fraisse family of graphs such that for every $G\in\mathcal G$ and $a,b,c\in G$ pairwise distinct with $\langle a,b\rangle,\langle b,c\rangle\in E(G)$, there is a graph $H\in\mathcal G$ and an epimorphism $f^H_G\colon H\to G$ so that whenever $p,q,r\in H$ are such that $f^H_G(p)=a,f^H_G(q)=b$ and $f^H_G(r)=c$, we have $\langle p,q\rangle\not\in E(H)$ or $\langle q,r\rangle \not\in E(H)$. Then, if $\G$ denotes the projective \Fraisse limit of $\mathcal G$, for every $a\in\G$ there is at most one $b\in\G\setminus\{a\}$ with $\langle a,b\rangle\in E(\G)$. In particular, the edge relation of $\G$ is an equivalence relation, that is, $\G$ is a prespace.
\end{thm}
\begin{proof}
Suppose on the contrary that there are three distinct vertices $a,b,c\in \mathbb G$ such that $\langle a,b\rangle, \langle b,c\rangle\in E(\mathbb G)$. Let a graph $G\in\mathcal G$ and an epimorphism $f_G\colon \mathbb G\to G$ be such that $f_G(a), f_G(b), \text{  and }f_G(c)$ are three distinct vertices of $G$. Then $\langle f_G(a),f_G(b)\rangle, \langle f_G(b),f_G(c)\rangle\in E(G)$, and thus, by our assumption, there is a graph
$H$ and an epimorphism $f^H_G\colon  H\to G$ such that for an epimorphism $f_H\colon \mathbb G\to H$ satisfying $f_G=f^H_G\circ f_H$ we have  $\langle f_H(a),f_H(b)\rangle\notin E(H)$ or $\langle f_H(b),f_H(c)\rangle\notin E(H)$. This contradicts the fact that $f_H$ maps edges to edges.
\end{proof}

\begin{defn}\label{def: splitting edges}
	Let $\T$ be a projective \Fraisse family of finite graphs. We say that $\T$ \emph{allows splitting edges} if for all $G\in\T$ and all distinct $a,b\in G$ with $\langle a,b\rangle\in E(G)$ the graph $H$ defined by $V(H)=V(G)\sqcup\{\ast\}$ and $$E(H)=(E(G)\setminus\{\langle a,b\rangle\})\cup\{\langle a,\ast\rangle,\langle\ast, b\rangle\}$$ and the two morphisms $H\to G$ that map $\ast$ to either $a$ or $b$ and are the identity otherwise, are in $\T$.
\end{defn}
Allowing to split edges is a basic closure property that will be satisfied by all the projective \Fraisse families considered in what follows. Despite its simplicity it has important consequences, such as the following lemma, and it will be used for some arguments in the next sections.
\begin{lemma}\label{lemma: splitting implies prespace}
	Suppose that $\T$ is a projective \Fraisse family of finite graphs that allows splitting edges. If $\G$ is the projective \Fraisse limit of $\T$, then $\G$ is a prespace.
\end{lemma}
\begin{proof}
	We will show that $\T$ satisfies the hypothesis of Theorem \ref{thm: transitive edges}, from which the conclusion follows immediately. Fix $G\in\T$ and pairwise distinct $a,b,c\in G$ with $\langle a,b\rangle,\langle b,c\rangle\in E(G)$ and let $H$ be the graph obtained from $G$ by splitting the edge $\langle a,b\rangle$ twice. More formally we have $V(H)=V(G)\sqcup\{a',b'\}$ and $$E(H)=(E(G)\setminus\{\langle a,b\rangle\})\cup\{\langle a,a'\rangle,\langle a',b'\rangle,\langle b',b\rangle\},$$ with the epimorphism $f^H_G$ such that $f^H_G(a')=a$, $f^H_G(b')=b$ and $f^H_G$ is the identity otherwise. We have that $H$ and $f^H_G$ are in $\T$ since $\T$ allows splitting edges, but if $p,q,r\in H$ are such that $f^H_G(p)=a$, $f^H_G(q)=b$ and $f^H_G(r)=c$, then either $\langle p,q\rangle\not\in E(H)$, or $\langle q,r\rangle\not\in E(H)$. 
\end{proof}

\subsection{Dendrites}

We begin this section by recalling some facts concerning dendrites from general topology. 
A compact connected metrizable space is called a \emph{dendrite} if it is 
locally connected and contains no simple closed curve \cite[Definition 10.1]{Nadler}.
Equivalently dendrites are locally connected dendroids \cite[Page 218]{Nadler_hyperspaces}.
For more characterizations of dendrites see \cite[Section 10.1]{Nadler}. Therefore a dendrite $X$ is uniquely arcwise connected, that is between any $x\neq y\in X$, there is a unique arc. In a dendrite $X$ there are three types of points $x\in X$, based on how many connected components $X\setminus\{x\}$ has:
\begin{itemize}
	\item If $X\setminus \{x\}$ is connected, then $x$ is called an \emph{endpoint} of $X$.
	\item If $X\setminus \{x\}$ has two connected components, then $x$ is called a \emph{regular point} of $X$.
	\item If $X\setminus\{x\}$ has more than two connected components, then $x$ is called a \emph{ramification point} of $X$.
\end{itemize}
In all of those cases the (potentially infinite) number of connected components of $X\setminus\{x\}$ is called the \emph{order} or \emph{degree} of $x$, denoted by $\ord(x)$, and it coincides with the \emph{Menger-Urysohn order} of $x$ in $X$, that is with the maximum number of arcs in $X$ meeting only in $x$ \cite[§46, I]{Kuratowski}. For $P\subseteq \{3,4,5,\ldots,\omega\}$ the \emph{generalized \Wazewski dendrite} $W_P$ is a dendrite satisfying the following two properties:
\begin{itemize}
	\item If $x\in W_P$ is a ramification point, $\ord(x)\in P$.
	\item If $p\in P$, the set of ramification points in $W_P$ of order $p$ is \emph{arcwise dense}. This means that for all $x,y\in W_P$ distinct, there is a ramification point of order $p$ in the unique arc joining $x$ and $y$.
\end{itemize}
For a fixed $P$ those two properties characterize a unique dendrite up to homeomorphism \cite[Theorem 6.2]{CD}, so we can actually talk about \emph{the} generalized \Wazewski dendrite $W_P$.

These notions were adapted to topological graphs in \cite{CR}, from which the remaining definitions and results of this section are taken.

\begin{defn}
	A topological graph $A$ is called \emph{arcwise connected} if for all $a,b\in A$, there is a subgraph of $A$ containing $a$ and $b$ which is an arc. A topological graph $A$ is called \emph{locally connected} if every $a\in A$ has an open neighbourhood which is connected. 
 A topological graph $G$ is {\it hereditarily unicoherent} if for every two non-empty closed connected topological graphs, $P$ and $Q$, with $V(P) \subseteq V(G)$, $V(Q) \subseteq V(G)$, $E(P) \subseteq E(G)$, and $E(Q) \subseteq E(G)$, the intersection $P\cap Q=(V(P)\cap V(Q), E(P)\cap E(Q))$  is connected. The topological graph $G$ is {\it unicoherent} if the above holds for all $P$ and $Q$ such that $V(P)\cup V(Q)=V(G)$.
 
\end{defn}

\begin{defn}
	A topological graph $G$ is called a \emph{graph-dendroid} if it is hereditarily unicoherent and arcwise connected. A graph-dendroid $G$ is called a \emph{graph-dendrite} if it is locally connected. 
\end{defn}

\begin{defn}\label{def: endpoint in topological graph}
	Let $G$ be a topological graph. A vertex $x\in G$ is called an \emph{endpoint} of $G$ if whenever $H$ is a topological graph which is an arc and $f\colon H\to G$ is an embedding with $x\in f(H)$, then $x$ is in the image of an endpoint of $H$. Note that, when $G$ is an arc, this notion agrees with the one in the Definition \ref{defn: arcs and more}.
\end{defn}

\begin{defn}
Let $X$ be a graph-dendrite. We say that $x\in X$ is a \emph{ramification point} if $X\setminus\{x\}$ has at least 3 components, in which case we call the (possibly infinite) number of components of $X\setminus\{x\}$ the \emph{order} of $x$ in~$X$.
\end{defn}

\begin{thm}[\protect{\cite[Lemma 3.7]{CR}}]\label{thm: monotone maps preserve endpoints} If $f\colon G\to H$ is a monotone epimorphism between topological graphs and $G$ is an arc, then $H$ is an arc and the images of endpoints of $G$ are endpoints of $H$. 
\end{thm}

\begin{proof}
 Denote the end vertices of $G$ by $a$ and $b$. We need to show that every vertex in $H\setminus \{f(a),f(b)\}$ disconnects $H$. Let $y\in H\setminus \{f(a),f(b)\}$; then, since
$G$ is an arc, the graph $G\setminus f^{-1}(y)$ is disconnected. Let $G\setminus f^{-1}(y)$ be the union of two disjoint graphs $G\setminus f^{-1}(y)=U\cup V$. Thus $H\setminus \{y\}=f(U)\cup f(V)$ and $f(U)\cap f(V)=\emptyset$, so $H\setminus \{y\}$ is disconnected as needed.
\end{proof}

Our goal is to construct a large class of dendrites from projective \Fraisse limits of finite graphs. The following results from \cite{CR} will be crucial. Lemma \ref{thm: limit is arcwise connected} is only stated in \cite{CR}, below we supplement a proof.

\begin{lemma}[\protect{\cite[Lemma 3.8]{CR}}] \label{thm: limit is arcwise connected}
Let $\Bbb T$ be the inverse limit of finite arcs with monotone epimorphisms. Then $\Bbb T$ is an arc. 
\end{lemma}
\begin{proof}
    Let $\langle I_n\mid n<\omega\rangle$ with monotone epimorphisms $f^m_n\colon I_m\to I_n$ be such that $\Bbb T=\varprojlim I_n$, and let $f_n\colon\Bbb T\to I_n$ be the induced 
    projections. Note that, since monotone epimorphisms map endpoints to endpoints, there are two points $a=(a_n)$ and $b=(b_n)$ in $\Bbb T$ such that $a_n,b_n$ are endpoints of $I_n$ for every $n$. In particular for any $x\in\Bbb T\setminus\{a,b\}$ 
    there is $N<\omega$ such that $f_k(x)\not\in\{a_k,b_k\}$ for all $k\geq N$. We want to show that every $x\in\Bbb T\setminus\{a,b\}$ disconnects $\Bbb T$. Let $N$ be as above for $x$ and let, for $k\geq N$, $U^k_a$ and $U^k_b$ be the 
    connected components of $I_k\setminus\{f_k(x)\}$ containing $a_k$ and $b_k$ respectively. Since $f_k$ is monotone, $f_k^{-1}(U^k_a)$ and $f_k^{-1}(U^k_b)$ are connected. Moreover for $k'\geq k$ we have $f_k^{-1}(U^k_a)\subseteq f_{k'}^{-1}(U^{k'}_a)$ (and analogously for $b$) so that $$\Bbb T\setminus\{x\}=\bigcup_{k\geq N}f_k^{-1}(U^k_a)\sqcup\bigcup_{k\geq N} f_k^{-1}(U^k_b)$$ is disconnected.
\end{proof}

\begin{thm}[\protect{\cite[Theorem 2.18]{CR}}]\label{thm: limit is unicoherent}
Let $\Bbb G$ be the inverse limit 
    of trees with monotone epimorphisms. Then $\Bbb G$ is hereditarily unicoherent.
\end{thm}
The proof below is an adaptation of the proof in Nadler \cite[Theorem 10.36]{Nadler} of the fact that the inverse limit of dendrites is hereditarily unicoherent.It is different from the proof given in \cite{CR}.
\begin{proof}
    Let $P,Q$ be closed connected subgraphs with $V(P)\subseteq V(\Bbb G)$, $V(Q)\subseteq V(\Bbb G)$, $E(P)\subseteq E(\Bbb G)$ and $E(Q)\subseteq E(\Bbb G)$. We want to show that $C=P\cap Q$ is connected. By assumption $\Bbb G=\varprojlim G_n$ is an inverse limit of trees with monotone epimorphisms $f^{n+1}_n\colon G_{n+1}\to G_n$. Let $P_n=f_n(P)$, $Q_n=f_n(Q)$ and $C_n=P_n\cap Q_n$, where $f_n\colon\Bbb G\to G_n$ is the canonical projection. If $C_n=\varnothing$ for some $n$ there is nothing to prove, so we can assume that $C_n\neq\varnothing$ for all $n$. Since every $G_n$ is a tree, hence hereditarily unicoherent, $C_n$ is connected and nonempty for every $n$. 
    It is now not hard to check that 
    $P\cap Q=\varprojlim f_n(P\cap Q)=
    \varprojlim (f_n(P)\cap f_n(Q))$, i.e. that
    $C=\varprojlim C_n$.
    This  implies that $C$ is connected, since it is the inverse limit of closed connected sets.
\end{proof}

\begin{thm}[\protect{\cite[Corollary 3.11+Observation 2.16]{CR}}] \label{thm: limit is dendrite}Let $\mathcal T$ be a projective \Fraisse family of trees with monotone epimorphisms, let $\G$ be its projective \Fraisse limit, then $\G$ is a graph-dendrite. Moreover if $\G$ has a transitive set of edges, then $|\G |$ is a dendrite.
\end{thm}
\begin{proof}
It is already proved in \cite[Proposition 2.1]{A_combinatorial_model} that $\mathbb{G}$ is connected and locally connected, and it is proved in \cite[Theorem 2.1]{A_combinatorial_model} that $|\mathbb{G}|$ is a Peano continuum (i.e. it is a locally connected compact connected space).

Peano continua are arcwise connected \cite[Theorem 8.23]{Nadler}.  It follows from Lemma \ref{thm: limit is arcwise connected} that $\mathbb G$ is arcwise connected (see also \cite[Proposition 3.10]{CR}), and it follows from Theorem \ref{thm: limit is unicoherent} that $\mathbb G$ is hereditarily unicoherent. The argument by contradiction 
can be used to show that $|\Bbb G|$ is hereditarily unicoherent whenever $\Bbb G$ is, which concludes the proof.
\end{proof}
Note that combining Lemma \ref{lemma: splitting implies prespace} and Theorem \ref{thm: limit is dendrite} we have that all projective \Fraisse families of trees with monotone maps that allow splitting edges have as limit a prespace whose topological realization is a dendrite. In particular this will be true for all projective \Fraisse families considered in the following sections. The  result below, which is contained in \cite[Proposition 2.1]{A_combinatorial_model}, 
will also be important.

\begin{lemma} 
\label{lemma: maps from the limit are monotone}
 Let $\T$ be a projective \Fraisse family of graphs with monotone epimorphisms. Let $\langle  T_i\mid i<\omega\rangle$ with maps $f^{i+1}_i\colon T_{i+1}\to T_i$ be a \Fraisse sequence for $\T$ and let $\mathbb T$ be its projective \Fraisse limit. Then for $G\in\T$ all epimorphisms $\mathbb T\to G$ in $\T^\omega$ are monotone. 
\end{lemma}
\begin{proof}
Since epimorphisms $f^{i+1}_i$ are monotone, it is not hard to see that the projection epimorphisms $f^i\colon\mathbb{T}\to T_i$  are monotone. Any epimorphism  $\mathbb T\to G$ is a composition of a projection epimorphism $f_i$, for some $i$, and a monotone epimorphism from $T_i$ to $G$.
\end{proof}

\section{Weakly Coherent Epimorphisms and Generalized \Wazewski Dendrites.}
\noindent In this section, we introduce a new class of maps between trees, which we call \emph{(weakly) coherent} and study the relationship between weakly coherent points in a prespace and ramification points of its topological realization. Once this is done we will show how to construct many generalized \Wazewski dendrites as the topological realization of a projective \Fraisse limit.

\subsection{Weakly Coherent Epimorphisms}
\begin{defn}
	Let $A,B$ be finite trees and $f\colon B\to A$ a monotone epimorphism. Fix $a\in A$  with $\ord(a)=n\geq 3$ and enumerate as $A_0,\ldots,A_{n-1}$ the connected components of $A\setminus\{a\}$. We say that $a$ is a \emph{point of weak coherence} for $f$ if there exists $b\in f^{-1}(a)$ such that $m=\ord(b)\geq n$ and an injection $p\colon n\to m$ such that if $B_0,\ldots,B_{m-1}$ are the connected components of $B\setminus\{b\}$, then $f^{-1}(A_i)\subseteq B_{p(i)}$ for all $0\leq i\leq n-1$. In this case we call $b$ the \emph{witness} for the weak coherence of $f$ at $a$. We say that $f$ is \emph{weakly coherent} if every $a\in A$ with $\ord(a)\geq 3$ is a point of weak coherence for $f$.
	
	This notion can be strengthened by requiring $\ord(b)=\ord(a)$ and $p$ to be a bijection. In that case we say that $a$ is a \emph{point of coherence of $f$} and that $b$ is a \emph{witness} for the coherence of $f$ at $a$. We say that $f$ is \emph{coherent} if every $a\in A$ with $\ord(a)\geq 3$ is a point of coherence for $f$. 
\end{defn}

\begin{remark}
   Definitions of a weakly coherent and of a coherent epimorphism can be extended (in the obvious way) to epimorphisms between inverse limits of trees. 
\end{remark}

\begin{remark}\label{remark: witnesses of coherence}
	Let $A,B$ be inverse limits of trees, $f\colon B\to A$ a monotone epimorphism and $a\in A$ a ramification point. There exists at most one $b\in B$ witnessing the weak coherence of $f$ at $a$. Consider now a third tree $C$ with a monotone epimorphism $g\colon C\to B$. If $f\circ g$ is weakly coherent at $a$ with witness $c$, then $f$ is weakly coherent at $a$. Then $b=g(c)$ witnesses the weak coherence of $f$ at $a$, and $c$ witnesses the weak coherence of $g$ at $b$.
 If $f\circ g$ and $g$ are coherent, then $f$ is coherent.
\end{remark}

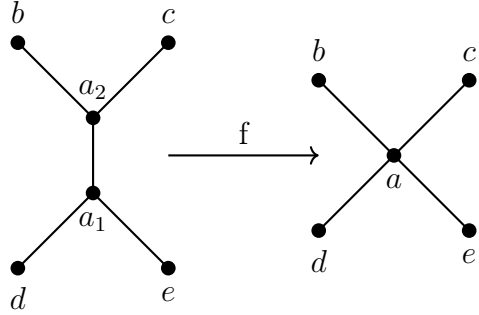
\begin{example}\noindent\label{example: monotone but not coherent}
	\begin{center}
		\begin{tikzpicture}[thick]
		\draw (-1,2) -- (0,1) (1,2) -- (0,1) (0,1) -- (0,0) 
		(-1,-1) -- (0,0) (1,-1) -- (0,0);
		
		\node [fill,circle,scale=0.5,label=below:$a_1$] at (0,0) {};
		\node [fill,circle,scale=0.5,label=$b$] at (-1,2) {};
		\node [fill,circle,scale=0.5,label=$c$] at (1,2) {};
		\node [fill,circle,scale=0.5,label=$a_2$] at (0,1) {};
		\node [fill,circle,scale=0.5,label=below:$d$] at (-1,-1) {};
		\node [fill,circle,scale=0.5,label=below:$e$] at (1,-1) {};
		
		\draw (4,0.5) -- (3,1.5) (4,0.5) -- (5,1.5) (4,0.5) -- (3,-0.5)
		(4,0.5) -- (5,-0.5);

		\draw [->] (1,0.5) -- (3,0.5) node[midway,above] {f};
		\node [fill,circle,scale=0.5,label=below:$a$] at (4,0.5) {};
		\node [fill,circle,scale=0.5,label=$b$] at (3,1.5) {};
		\node [fill,circle,scale=0.5,label=$c$] at (5,1.5) {};
		\node [fill,circle,scale=0.5,label=below:$d$] at (3,-0.5) {};
		\node [fill,circle,scale=0.5,label=below:$e$] at (5,-0.5) {};
		\end{tikzpicture}
	\end{center}
	The map $f$ defined by $f(a_1)=f(a_2)=a$ and mapping every other node to the node with the same name is a monotone epimorphism which is not weakly coherent at $a$.
\end{example}

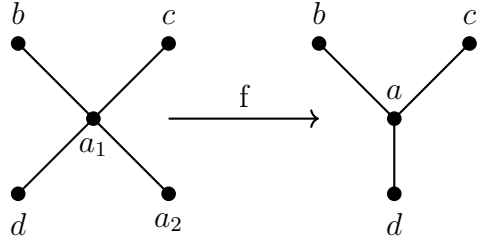
\begin{example}\noindent
	\begin{center}
		\begin{tikzpicture}[thick]
		\draw (-1,1) -- (0,0) (1,1) -- (0,0)
		(-1,-1) -- (0,0) (1,-1) -- (0,0);
		
		\node [fill,circle,scale=0.5,label=below:$a_1$] at (0,0) {};
		\node [fill,circle,scale=0.5,label=$b$] at (-1,1) {};
		\node [fill,circle,scale=0.5,label=$c$] at (1,1) {};
		\node [fill,circle,scale=0.5,label=below:$d$] at (-1,-1) {};
		\node [fill,circle,scale=0.5,label=below:$a_2$] at (1,-1) {};
		
		\draw (4,0) -- (3,1) (4,0) -- (5,1) (4,0) -- (4,-1);

		\draw [->] (1,0) -- (3,0) node[midway,above] {f};
		\node [fill,circle,scale=0.5,label=$a$] at (4,0) {};
		\node [fill,circle,scale=0.5,label=$b$] at (3,1) {};
		\node [fill,circle,scale=0.5,label=$c$] at (5,1) {};
		\node [fill,circle,scale=0.5,label=below:$d$] at (4,-1) {};
		
		\end{tikzpicture}
	\end{center}
	The map $f$ defined by $f(a_1)=f(a_2)=a$ and mapping every other node to the node with the same name is a weakly coherent epimorphism which is not coherent at $a$.
\end{example}

\begin{example}\label{example: order three}
	Let $A,B$ be finite trees in which every vertex has degree at most $3$. Then any monotone epimorphism $f\colon A\to B$ is coherent. Indeed fix $b\in B$ of order $3$ and let $b_1$,$b_2$ and $b_3$ be its distinct neighbours. Let $b'_1\in f^{-1}(b_1),b'_2\in f^{-1}(b_2)$ and let $P=(p_0=b_1',p_1,\ldots,p_n=b'_2)$ be the unique arc from $b'_1$ to $b'_2$. By Remark \ref{remark: connected image}, $f(P)$ is connected and since it contains $b_1,b_2$ and $B$ is uniquely arcwise connected it must also contain $b$. Hence $P\cap f^{-1}(b)\neq\varnothing$. Let $b'_3\in f^{-1}(b_3)$ and let $Q$ be the shortest arc from $b'_3$ to a point in $P$, let $p_i=P\cap Q$ and note that, by another application of Remark \ref{remark: connected image}, similar to the one above, $p_i\in f^{-1}(b)$. It is now easy to check that $p_i$ witnesses the coherence of $f$ at $b$.
\end{example}

\begin{remark}\label{remark: composition of coherent is coherent}
	Note that if $A,B,C$ are finite trees and $f\colon A\to B,\,g\colon B\to C$ are (weakly) coherent epimorphisms, then $g\circ f\colon A\to C$ is also (weakly) coherent. Indeed if $c\in C$ with $\ord(c)\geq 3$ then there is some $b\in g^{-1}(c)$ witnessing the (weak) coherence of $g$ at $c$ and some $a\in f^{-1}(b)$ witnessing the (weak) coherence of $f$ at $b$. It is easy to check that $a$ also witnesses the (weak) coherence of $g\circ f$ at $c$.
\end{remark}
\begin{remark}
	Note also that if $A$ is a finite tree with at least two vertices and $B$ with $V(B)=V(A)\sqcup\{\ast\}$ obtained from $A$ by splitting an edge as in Definition \ref{def: splitting edges}, then the two maps $B\to A$ mapping $\ast$ to either endpoint of the split edge are coherent.
\end{remark}

\begin{defn}
	Given an inverse system $\langle A_i,i<\omega\rangle$ of graphs and monotone epimorphisms $f^i_j\colon A_i\to A_j$ we say that $a=(a_i)_i\in\varprojlim A_i\subseteq\prod A_i$ is a \emph{point of weak coherence} or a \emph{weakly coherent point} if there is a $k\in\omega$ such that for all $l>k$, $a_l$ is a point of weak coherence of $f^{l+1}_l$ and $a_{l+1}$ witnesses this. We say that $a$ is a \emph{point of coherence} or a \emph{coherent point} if there is a $k\in\omega$ such that for all $l>k$, $a_l$ is a point of coherence of $f^{l+1}_l$, and $a_{l+1}$ witnesses this.
\end{defn}

\begin{lemma}\label{lemmaw: uniqueness}
	Let $\T$ be a projective \Fraisse family of finite trees with monotone epimorphisms that allows splitting edges, let $\G$ be its projective \Fraisse limit and let $\pi\colon\G\to |\G|$ be its topological realization, which is a dendrite by Theorem \ref{thm: limit is dendrite}. Moreover let $p\in |\G|$ be a ramification point. Then there is a unique $p'\in\G$ with $\pi(p')=p$.
\end{lemma}
\begin{proof}
	Suppose for a contradiction that $|\pi^{-1}(p)|>1$ and, using that equivalence classes in $\G$ have at most two elements by Theorem \ref{thm: transitive edges}, write $\pi^{-1}(p)=\{p_1,p_2\}$. Let $\{A_i'\mid i\leq n\}$ be the connected components of $|\Bbb G|\setminus\{p\}$, where $\ord(p)=n\in\Bbb N\cup\{\omega\}$. Let $A_i=A_i'\cup\{p\}$ for $i=1,\ldots,n$  and note that $\{\pi^{-1}(A_i)\mid i\leq n\}$ are connected sets in $\G$ intersecting only in $\{p_1,p_2\}$. Find a finite tree $G$ and a monotone epimorphism $f_G\colon \G\to G$ such that $f_G(\pi^{-1}(A_i))$ is nontrivial (meaning not contained in $\{f_G(p_1),f_G(p_2)\}$) for $i\in\{1,2,3\}$ and such that $f_G(p_1)\neq f_G(p_2)$. This is possible because epimorphisms in $\T^\omega$ are monotone by Lemma \ref{lemma: maps from the limit are monotone}. We claim that $$f_G(\pi^{-1}(A_i))\cap f_G(\pi^{-1}(A_j))\subseteq\{f_G(p_1),f_G(p_2)\},$$ for all distinct $i,j\in\{1,2,3\}$. Suppose for a contradiction that $$r\in f_G(\pi^{-1}(A_1))\cap f_G(\pi^{-1}(A_2))\setminus\{f_G(p_1),f_G(p_2)\}.$$ Then $f_G^{-1}(r)$ and $\{p_1,p_2\}$ are disjoint, nonempty, closed, connected subsets of $\Bbb G$. However $f_G^{-1}(r)$ meets both $\pi^{-1}(A_1)\setminus\{p_1,p_2\}$ and $\pi^{-1}(A_2)\setminus\{p_1,p_2\}$, which is a contradiction since $f_G^{-1}(r)$ is connected, while \\ $(\pi^{-1}(A_1)\cup\pi^{-1}(A_2))\setminus\{p_1,p_2\}$ is not. 
	
	The idea now is that we can always move to a bigger graph $H$ obtained from $G$ by splitting an edge in order to force $f_H(\pi^{-1}(A_1))\cap f_H(\pi^{-1}(A_3))\not\subseteq\{f_H(p_1),f_H(p_2)\}$. Indeed there must be at least two of $f_G(\pi^{-1}(A_1))$, $f_G(\pi^{-1}(A_2))$, $f_G(\pi^{-1}(A_3))$ that are connected to the same vertex $f_G(p_1)$ or $f_G(p_2)$, suppose without loss of generality that $f_G(\pi^{-1}(A_1))$ and $f_G(\pi^{-1}(A_3))$ are both connected to $f_G(p_2)$. Now construct a graph $H$ by taking $H=G\sqcup\{r\}$, where $r$ is a vertex not in $G$ and $$E(H)=(E(G)\setminus\{\langle f_G(p_1),f_G(p_2)\rangle\})\cup\langle f_G(p_1),r\rangle,\langle r,f_G(p_2)\rangle,$$ together with the epimorphism $f^H_G\colon H\to G$ defined by $f^H_G(r)=f_G(p_2)$ and the identity otherwise. Call $a=(f^H_G)^{-1}(f_G(p_1))$ and $b$ the vertex in $(f^H_G)^{-1}(f_G(p_2))$ distinct from $r$. Note that $H,f^H_G\in\T$ since $\T$ allows splitting edges. By definition of projective \Fraisse limit there must be an epimorphism $f_H\colon\Bbb G\to H$ such that $f^H_G\circ f_H=f_G$. The last identity combined with the fact that $f_H$ maps edges to edges forces $f_H(p_1)=a$ and $f_H(p_2)=r$. Now we can use the same argument as above to show that $$f_H(\pi^{-1}(A_1))\cap f_H(\pi^{-1}(A_3))\subseteq\{f_H(p_1),f_H(p_2)\}.$$ This is a contradiction, since by construction those two sets also meet in $b$.
\end{proof}

\begin{thm}\label{thmw: coherent implies splitting}
	Let $\T$ be a projective \Fraisse family of trees with monotone epimorphisms that allows splitting edges, let $\langle G_i\mid i<\omega\rangle$ with epimorphisms $f^m_n\colon G_m\to G_n$ be a \Fraisse sequence for $\T$, let $\G=\varprojlim G_i$ be the projective \Fraisse limit of $\T$ and let $\pi\colon \G\to |\G|$ be its topological realization, which is a dendrite by Theorem \ref{thm: limit is dendrite}. Then $\pi$ maps points of weak coherence of $\G\subseteq\prod G_i$ to ramification points of $|\G|$. 
\end{thm}
\begin{proof}
	Suppose that $p=(p_i)$ is a point of weak coherence of $\G$, we want to show that $\pi(p)$ is a ramification point of $|\G|$. Let $f_k\colon\G\to G_k$ be the projection epimorphisms obtained by construction of the inverse limit, let $A^1_k,\ldots,A^{n_k}_k$ be the connected components of $G_k\setminus\{p_k\}$, enumerated so that $(f^{k+1}_k)^{-1}(A^i_k)\subseteq A^i_{k+1}$ for all $i$. Let $k(i)$ be the least $k$ such that $G_k\setminus\{p_k\}$ has at least $i$ connected components, if such an $i$ exists, and $k(i)=\omega$ otherwise. Note that $p=\bigcap f_k^{-1}(p_k)$ and that $$\{A^i=\bigcup_{k\geq k(i)} f_k^{-1}(A^i_k)\mid 1\leq i< \omega\}$$ are the connected components of $\G\setminus \{p\}$ (with $A^i$ being empty if $k(i)=\omega$). We will show that the equivalence class of $p$ in $\G$ is a singleton, from which we can conclude that $\pi(A^i)$ are the connected components of $|\G|\setminus\{\pi(p)\}$, so that $\pi(p)$ is a ramification point of $|\G|$ of order equal to the least $n$ such that $k(n+1)=\omega$ or to $\omega$ if there is no such $n$. Suppose for a contradiction that there exist $q=(q_i)\in\G$ with $\langle p,q\rangle\in E(\G)$, so in particular $\langle p_i,q_i \rangle \in E(G_i)$ for all $i$. Let $k$ be big enough so that $p_k\neq q_k$ and let $A^j_k$ be the component of $G_k\setminus \{p_k\}$ containing $q_k$. Note that by weak coherence for all $m\geq k$, $q_m\in A^j_m$. Consider the tree $H$ obtained from $G_k$ by splitting the $\langle p_k,q_k\rangle$ edge into two edges $\langle p_k,p_k'\rangle$ and $\langle p_k',q_k\rangle$, with the epimorphism $f^H_k\colon H\to G_k$ such that $f^H_k(p'_k)=p_k$ and the identity otherwise. Note that $f^H_k$ is coherent at $p_k\in G_k$, as witnessed by $p_k\in H$. 
	By definition of \Fraisse sequence we can find $l$ big enough and an epimorphism $f^l_H\colon G_l\to H$ such that $f^l_k=f^H_k\circ f^l_H$, which implies that $f^l_H(q_l)=q_k\in H$. On one hand we now must have $f^l_H(p_l)=p_k\in H$ by weak coherence, since $p_l$ witnesses the weak coherence of $f^l_k$ at $p_k\in G_k$ and $p_k\in H$ witnesses the weak coherence of $f^H_k$ at $p_k\in G_k$, on the other hand we must have $f^l_H(p_l)=p_k'$, since $f^l_H$ maps edges to edges, a contradiction. 
\end{proof}

\begin{remark}\label{rmk: order of weakly coherent}
	Note in particular that if $x=(x_i)$ is a point of weak coherence in $\prod G_i$ then either $\ord(x_i)$ stabilizes after some index $j\in\Bbb N$ to some finite value $n$ or it grows unboundedly. In the former case $\pi(x)$ also has order $n$ as a ramification point of $|\Bbb G|$, while in the latter $\pi(x)$ is a ramification point of $|\Bbb G|$ of infinite order.
\end{remark}
	
The converse to Theorem \ref{thmw: coherent implies splitting} is false in general, but it is true when monotone epimorphisms are replaced with weakly coherent ones, as shown in the following theorem.

\begin{thm}\label{thmw: splitting implies coherent}
	Let $\T$ be a projective \Fraisse family of trees with weakly coherent epimorphisms, let $\langle G_i\mid i<\omega\rangle$ with epimorphisms $f^m_n\colon G_m\to G_n$ be a \Fraisse sequence for $\T$ and let $\pi\colon\G\to|\G|$ be its topological realization, which is a dendrite by theorem \ref{thm: limit is dendrite}. Then $\pi^{-1}$ maps ramification points of $|\G|$ to points of weak coherence of $\G\subseteq\prod G_i$. 
\end{thm}	
\begin{proof}
	As above let $f_k\colon\G\to G_k$ be the projection epimorphisms obtained by construction of the limit and let $x'\in|\G|$ be a ramification point. By Lemma \ref{lemmaw: uniqueness} there is a unique $x=(x_i)\in\pi^{-1}(x')$ and we want to show that it is a point of weak coherence in $\prod G_i$. Let $\{A^i\mid i<\ord(x)\}$ with $\ord(x)\in\Bbb N\cup\{\omega\}$ be the connected components of $\G\setminus\{x\}$ and let $k$ be big enough so that $f_k(A^1),f_k(A^2),f_k(A^3)$ are not singletons in $G_k$. Let $A^1_k,A^2_k$ and $A^3_k$ be connected components of $G_k\setminus \{x_k\}$, numbered so that $f_k^{-1}(A^i_k)\subseteq A^i$ for $i=1,2,3$, which is always possible since $f_k^{-1}(A^i_k)$ is connected by Lemma \ref{lemma: maps from the limit are monotone} and doesn't meet $f_k^{-1}(x_k)$ by construction. By assumption $f^{k+1}_k\colon G_{k+1}\to G_k$ is weakly coherent at $x_k$, as witnessed by some vertex $a\in G_{k+1}$. Suppose by contradiction that $x_{k+1}\neq a$. Note that if $A^i_{k+1}$ is a connected component of $G_{k+1}\setminus\{ x_{k+1}\}$, then there is a unique $j$ such that $f_{k+1}^{-1}(A^i_{k+1})\subseteq A^j$, once again because $f_{k+1}^{-1}(A^i_{k+1})$ is connected and doesn't meet $f^{-1}_{k+1}(x_{k+1})$.
	Since $x_{k+1}\neq a$ and $G_{k+1}$ is a tree, there is a connected component $A^n_{k+1}$ of $G\setminus \{x_{k+1}\}$, that contains all connected components of $G\setminus \{a\}$ except at most one, so it also contains $(f^{k+1}_k)^{-1}(A^i_k)$ for at least two distinct $i$. Assume without loss of generality that 
	\begin{align*}
	(f^{k+1}_k)^{-1}(A^1_k)&\subseteq A^n_{k+1} \\
	(f^{k+1}_k)^{-1}(A^2_k)&\subseteq A^n_{k+1},
	\end{align*} 
	and let $j$ be the unique index such that $f_{k+1}^{-1}(A^n_{k+1})\subseteq A^j$. But now we have both $f^{k+1}_k\circ f_{k+1}=f_k$ and 
	\begin{align*}
	(f_{k+1}^{-1}\circ(f^{k+1}_k)^{-1})(A^1_k)&\subseteq A^j\\
	(f_{k+1}^{-1}\circ(f^{k+1}_k)^{-1})(A^2_k)&\subseteq A^j\\
	f_k^{-1}(A^1_k)&\subseteq A^1\\
	f_k^{-1}(A^2_k)&\subseteq A^2,
	\end{align*}
	which cannot all hold at the same time regardless of the value of $j$, a contradiction.
\end{proof}

\subsection{Generalized \Wazewski Dendrites from Projective \Fraisse Families}\hfill\\
We now introduce some families of finite trees. The remainder of this section will be dedicated to showing that those are projective \Fraisse families and that the topological realizations of their projective \Fraisse limits are generalized \Wazewski dendrites.

\begin{defn}\label{def: families Fp and Gp}
	Let $P\subseteq\{3,4,5,\ldots,\omega\}$. We consider two cases. If $\omega\in P$ we consider the family $\mathcal F_P$ whose elements are finite trees with no vertices of order $2$. Given $A,B\in\mathcal F_P$, an epimorphism of graphs $f\colon B\to A$ is in $\mathcal F_P$ if:
	\begin{enumerate}
		\item $f$ is monotone;
		\item If $a\in A$ is such that $\ord(a)\in P$, then $f$ is coherent at $a$.
		\item If $a\in A$ is such that $\ord(a)\not\in P$ and $\ord(a)\geq 3$ then $f$ is weakly coherent at $a$, and if $b\in f^{-1}(a)$ is the witness for the weak coherence of $f$ at $a$, then $\ord(b)\not\in P$.
	\end{enumerate}

\noindent On the other hand, if $\omega\not\in P$, we consider the family $\mathcal G_P$ whose elements are finite trees all of whose vertices are either endpoints or have order in $P$, with coherent monotone epimorphisms.
\end{defn}

\begin{remark}\label{remark: splitting edges}
	The families $\F_P$ and $\mathcal G_P$ don't allow splitting edges as defined earlier, since the trees in those families don't have any vertices of degree two. However if $T$ is a tree in $\F_P$ or $\mathcal G_P$ and $a,b\in T$ with $\langle a,b\rangle\in E(T)$ it is still possible to split the edge $\langle a,b\rangle$ by removing it and then adding a new point $x$ connected to $a$ and $b$. We have to also add enough new neighbours $x_i$ to $x$ until $\ord(x)\in P$. This is enough for all the arguments in the previous sections that used the edge splitting property.
\end{remark}

We will now verify that $\mathcal F_P$ and $\mathcal G_P$ are projective \Fraisse families and later we will identify the topological realizations of their projective \Fraisse limits. It seems natural to guess that if $\mathbb F_P$ and $\G_P$ are the projective \Fraisse limits of $\F_P$ and $\mathcal G_P$, then $|\mathbb F_P|\cong W_P$ and $|\G_P|\cong W_P$, but unfortunately this is not always the case. It is true for all families of the form $\mathcal G_P$, but for families of the form $\F_P$ it only holds when $\{3,4,5,\ldots,\omega\}\setminus P$ is infinite, as we will see later. Clearly both $\mathcal F_P$ and $\mathcal G_P$ contain countably many structures up to isomorphism, contain the identity morphisms and their morphisms are closed under composition by Remark \ref{remark: composition of coherent is coherent}. We begin by verifying the joint projection property for $\F_P$ and $\mathcal G_P$, with a construction that works for both families.

\begin{lemma}\label{lemma: joint projection property}
	Let $A,B\in\F_P$ (respectively $A,B\in\mathcal G_P$). Then there exists $C\in\F_P$ (respectively $C\in\mathcal G_P$) with epimorphisms $f_1\colon C\to A,\,f_2\colon C\to B$ in $\F_P$ (respectively in $\mathcal G_P$). 
\end{lemma} 
\begin{proof}
	Let $a\in A,\,b\in B$ be two endpoints. Consider the tree $C'$ obtained by taking the disjoint union of $A$ and $B$ and identifying $a$ with $b$. Let $C$ be the tree obtained from $C'$ by adding new vertices $x_i$ connected only to $x$ until $\ord(x)\in P$, where $x$ is the vertex $a$ and $b$ have been identified in. Define $f\colon C\to A$ to be the identity on $V(A)$, while any other vertex is mapped to $x$. Analogously define $G\colon C\to B$ to be the identity on $V(B)$, while any other vertex is mapped to $x$. It is immediate to check that $C$ witnesses the joint projection property for $A,B$.
\end{proof}

It remains to verify that the projective amalgamation property is satisfied. We do so for families of the form $\mathcal F_P$ and note that the procedure described for those to produce an amalgam $D$ from a diagram $C\to A\leftarrow B$ of structures in $\mathcal F_P$ can also be applied to diagrams of the same shape in $\mathcal G_P$, and the amalgam $D$ produced in that case is itself an element of $\mathcal G_P$.

\begin{lemma}\label{lemma: amalgamation is satisfied}
	The family $\F_P$ satisfies the projective amalgamation property, so it is a projective \Fraisse family. 	Explicitly for all $A,B,C\in\F_P$ with epimorphisms $f\colon B\to A$ and $g\colon C\to A$ in $\F_P$, there exist $D\in\F_P$ with epimorphisms $h_1\colon D\to B$ and $h_2\colon D\to C$ in $\F_P$ such that $f\circ h_1=g\circ h_2$.
\end{lemma}
\begin{proof}
	We proceed by induction on $|E(B)|+|E(C)|$. If there are no nontrivial edges in $B$ and $C$ then $A$ must also be a singleton and we can take $D$ to be a singleton as well.  For the inductive step suppose now that $B$ and $C$ are not both singletons and that the lemma is proved for all diagrams in $\F_P$ of the form $F\to A\leftarrow G$ with $|E(F)|+|E(G)|<|E(B)|+|E(C)|$. If $A=B=C$ and $f=g=\mathrm{Id}_A$ we can simply take $D=A$ and $h_1=h_2=\mathrm{Id}_A$, so we can assume that at least one of $f$ and $g$ is nontrivial. Suppose without loss of generality that $f$ is nontrivial. Find $a\in A$ such that $|f^{-1}(a)|>1$. There are two possibilities now, $a$ could be a ramification point or and endpoint of $A$, and we deal with those two cases separately.  
	
	\textbf{Case 1: $\mathbf{a}$ is a ramification point.} Let $b\in B$ be the witness for the (weak) coherence of $f$ at $a$ and let $B_1$ be a component of $B\setminus\{b\}$ with $B_1\cap f^{-1}(a)\neq\varnothing$. Let $B_2=B_1\cap f^{-1}(a)$. We need to distinguish two cases, based on whether $B_2=B_1$ or not. 
	
	\textbf{Case 1a: $\mathbf{B_2\neq B_1}$} We start with the harder case, that is $B_2\neq B_1$ and explain how to deal with the easier case later. Let $b_1$ be the only element of $B_2$ with $\langle b,b_1\rangle\in E(B)$ and let $b_2$ be the only element of $B_2$ for which there exist a $b_3\in B_1\setminus B_2$ with $\langle b_2,b_3\rangle\in E(B)$ (uniqueness of $b_1,b_2,b_3$ follows from $B$ being a tree). Consider now the graph $B'$ obtained from $B$ by collapsing $B_2$ to $b$, formally $V(B')=V(B)\setminus V(B_2)$ and $$E(B')=(E(B)\setminus\left( E(B_2)\cup\{\langle b,b_1\rangle,\langle b_2,b_3\rangle\})\right)\cup\{\langle b,b_3\rangle\},$$ see Figure \ref{fig: amalgamation}. There is a natural map $\pi\colon B\to B'$ given by $\pi(x)=b$ if $x\in B_2$ and the identity otherwise. Note that $B'\in\F_P$ since both $b$ and $b_3$ have the same order as $\pi(b)$ and $\pi(b_3)$, and $\pi$ is a coherent epimorphism, so in particular it is in $\F_P$. Let $f'\colon B'\to A$ be the unique map with $f'\circ \pi=f$ and note that $f'$ is still  in $\F_P$ by construction. We reach the following picture, where the important vertices and edges are drawn, while the dashed sections represent parts of the trees whose precise structure is not important.
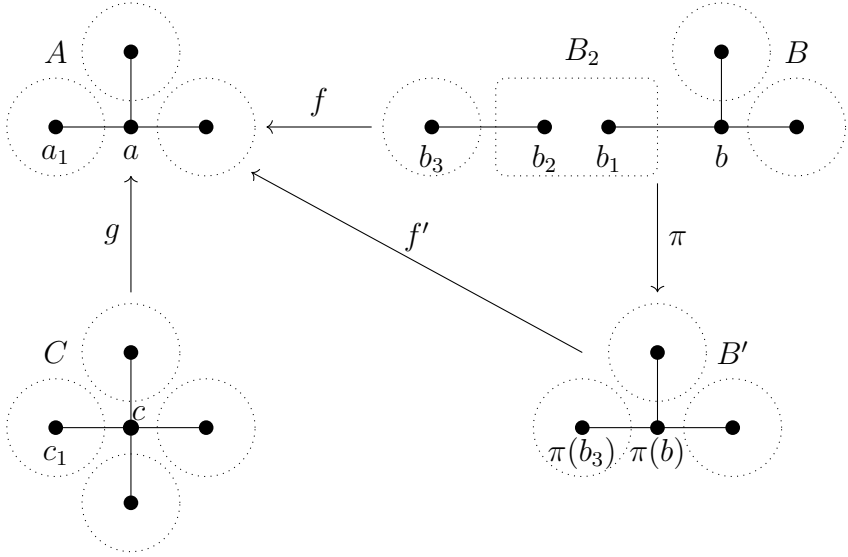
\begin{figure}
        \centering
		\begin{tikzpicture}
		\node at (0,1) {$A$};
		\draw [dotted] (0,0) circle (0.65cm);
		\draw [dotted] (2,0) circle (0.65cm);
		\draw [dotted] (1,1) circle (0.65cm);
		\draw (0,0) -- (2,0);
		\draw (1,0) -- (1,1);
		\node [fill,circle,scale=0.5,label=below:$a$] at (1,0) {};
		\node [fill,circle,scale=0.5,label=below:$a_1$] at (0,0) {};
		\node [fill,circle,scale=0.5] at (1,1) {};
		\node [fill,circle,scale=0.5] at (2,0) {};
		
		\draw [->] (4.2,0) -- (2.8,0) node[midway,above] {$f$};
		
		\node at (9.85,1) {$B$};
		\node at (7,1) {$B_2$};
		\draw [dotted] (5,0) circle (0.65cm);
		\draw (5,0) -- (6.5,0);
		\draw (7.35,0) -- (8.85,0);
		\draw[rounded corners,dotted] (5.85, -0.65) rectangle (8, 0.65) {};
		\node [fill,circle,scale=0.5,label=below:$b_3$] at (5,0) {};
		\node [fill,circle,scale=0.5,label=below:$b_2$] at (6.5,0) {};
		\node [fill,circle,scale=0.5,label=below:$b_1$] at (7.35,0) {};
		\node [fill,circle,scale=0.5,label=below:$b$] at (8.85,0) {};
		\draw (8.85,0) -- (9.85,0);
		\draw (8.85,0) -- (8.85,1);
		\draw [dotted] (8.85,1) circle (0.65cm);
		\draw [dotted] (9.85,0) circle (0.65cm);
		\node [fill,circle,scale=0.5] at (8.85,1) {};
		\node [fill,circle,scale=0.5] at (9.85,0) {};
		
		\draw [->] (1,-2.2) -- (1,-0.65) node[midway,left] {$g$};
		
		\node at (0,-3) {$C$};
		\draw [dotted] (0,-4) circle (0.65cm);
		\draw [dotted] (2,-4) circle (0.65cm);
		\draw [dotted] (1,-3) circle (0.65cm);
		\draw [dotted] (1,-5) circle (0.65cm);
		\draw (0,-4) -- (2,-4);
		\draw (1,-5) -- (1,-3);
		\node [fill,circle,scale=0.55,label={[xshift=0.1cm, yshift=-0.15cm]$c$}] at (1,-4) {};
		\node [fill,circle,scale=0.5,label=below:$c_1$] at (0,-4) {};
		\node [fill,circle,scale=0.5] at (1,-3) {};
		\node [fill,circle,scale=0.5] at (2,-4) {};
		\node [fill,circle,scale=0.5] at (1,-5) {};
		
		\node at (9,-3) {$B'$};
		\draw [dotted] (7,-4) circle (0.65cm);
		\draw (7,-4) -- (9,-4);
		\node [fill,circle,scale=0.5,label={[yshift=-0.8cm]$\pi(b_3)$}] at (7,-4) {};
		\node [fill,circle,scale=0.5,label={[yshift=-0.8cm]$\pi(b)$}] at (8,-4) {};
		\draw (8,-4) -- (8,-3);
		
		\draw [dotted] (8,-3) circle (0.65cm);
		\draw [dotted] (9,-4) circle (0.65cm);
		\node [fill,circle,scale=0.5] at (8,-3) {};
		\node [fill,circle,scale=0.5] at (9,-4) {};
		
		\draw [->] (8,-0.75) -- (8, -2.2) node[midway,right] {$\pi$};
		
		\draw [->] (7,-3) -- (2.6,-0.6) node[midway,above] {$f'$};
		\end{tikzpicture}
        \caption{The construction of $B'$ from $B$}
        \label{fig: amalgamation}
\end{figure}
	Since $|E(B')|<|E(B)|$ we can amalgamate $f'$ and $g$ over $A$ by inductive hypothesis, so we can find $D'\in\F_P$ and epimorphisms $h'_1\colon D'\to B'$ and $h'_2\colon D'\to C$ in $\F_P$ such that $f'\circ h_1'=g\circ h_2'$. Now we want to construct $D$ together with $h_1\colon D\to B$ and $h_2\colon D\to C$ from $D'$, $h'_1$ and $h'_2$, the intuitive idea is that we just need to paste back $B_2$ in $D'$ in the right spot, in particular $B_2$ should be inserted on the edge of $D'$ that corresponds to the edge $\langle \pi(b),\pi(b_3)\rangle$ in $B'$ (analogously to how $B$ is obtained from $B'$ by inserting $B_2$ on the $\langle \pi(b),\pi(b_3)\rangle$ edge). Let $a_1$ be the unique vertex of $A$ with $\langle a,a_1\rangle\in E(A)$ and $f^{-1}(a_1)\subseteq B_1$ and let $A_1$ be the connected component of $A\setminus \{a\}$ containing $a_1$. Let $c\in C$ witness the (weak) coherence of $g$ at $a$ and let $C_1$ be the connected component of $C\setminus \{c\}$ containing $g^{-1}(a_1)$. Let $c_1\in C_1$ be the unique vertex of $C$ with $\langle c,c_1\rangle\in E(C)$. Let $d\in D'$ witness the (weak) coherence of $h'_2$ at $c$ and let $D_1$ be the connected component of $D'\setminus \{d\}$ containing $(h'_2)^{-1}(c_1)$. Note that, since $f'\circ h'_1=g\circ h'_2$ and witnesses for (weak) coherence are unique, $d$ also witnesses the (weak) coherence of $h'_1$ at $b$ and $(h'_1)^{-1}(\pi(B_1))\subseteq D_1$ by Remark \ref{remark: witnesses of coherence}. Now let $d_1\in \{d\}\cup(D_1\cap (h'_1)^{-1}(\pi(b)))$ and $d_2\in D_1\cap (h'_1)^{-1}(\pi(b_3))$ be the unique such vertices with $\langle d_1,d_2\rangle\in E(D')$. We can now construct $D$ by inserting $B_2$ between $d_1$ and $d_2$, formally $V(D)=V(D')\sqcup V(B_2)$ and $$E(D)=(E(D')\setminus\{\langle d_1,d_2\rangle\})\cup E(B_2)\cup \{\langle b_1,d_1\rangle,\langle b_2,d_2\rangle\},$$
	note that $B_2\in\F_P$ and we didn't change the degrees of $d_1$ and $d_2$, so $D\in\F_P$ as well. Now consider the projection $\pi'\colon D\to D'$ collapsing $B_2$ to $d_1$, so $\pi'(x)=d_1$ if $x\in B_2$ and the identity otherwise, which is also an epimorphism in $\F_P$. We are now in a situation that looks like Figure \ref{fig: amalgamation2}.
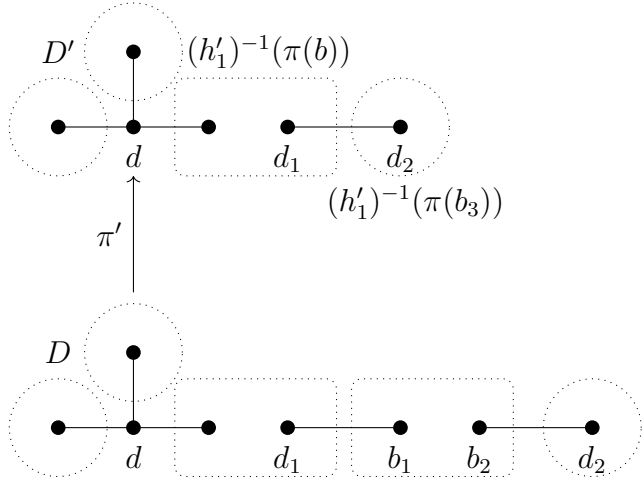
\begin{figure}
    \centering
		\begin{tikzpicture}
		\node at (0,1) {$D'$};
		\node at (2.8,1) {$(h_1')^{-1}(\pi(b))$};
		\node at (4.75,-1) {$(h_1')^{-1}(\pi(b_3))$};
		\draw [dotted] (0,0) circle (0.65cm);
		\draw [dotted] (1,1) circle (0.65cm);
		\draw (0,0) -- (2,0);
		\draw (1,0) -- (1,1);
		\node [fill,circle,scale=0.5,label=below:$d$] at (1,0) {};
		\draw[rounded corners,dotted] (1.55, -0.65) rectangle (3.7, 0.65) {};
		\node [fill,circle,scale=0.5] at (2,0) {};
		\node [fill,circle,scale=0.5,label=below:$d_1$] at (3.05,0) {};
		\node [fill,circle,scale=0.5] at (0,0) {};	
		\node [fill,circle,scale=0.5] at (1,1) {};
		\draw (3.05,0) -- (4.55,0);
		\node [fill,circle,scale=0.5,label=below:$d_2$] at (4.55,0) {};
		\draw [dotted] (4.55,0) circle (0.65cm);
		
		\node at (0,-3) {$D$};
		\draw [dotted] (0,-4) circle (0.65cm);
		\draw [dotted] (1,-3) circle (0.65cm);
		\draw (0,-4) -- (2,-4);
		\draw (1,-4) -- (1,-3);
		\node [fill,circle,scale=0.5,label=below:$d$] at (1,-4) {};
		\draw[rounded corners,dotted] (1.55, -4.65) rectangle (3.7, -3.35) {};
		\draw[rounded corners,dotted] (3.9, -4.65) rectangle (6.05, -3.35) {};
		\node [fill,circle,scale=0.5] at (2,-4) {};
		\node [fill,circle,scale=0.5,label=below:$d_1$] at (3.05,-4) {};
		\node [fill,circle,scale=0.5] at (0,-4) {};	
		\node [fill,circle,scale=0.5] at (1,-3) {};
		\draw (3.05,-4) -- (4.55,-4);
		\node [fill,circle,scale=0.5,label=below:$b_1$] at (4.55,-4) {};
		\node [fill,circle,scale=0.5,label=below:$b_2$] at (5.6,-4) {};
		\draw [dotted] (7.1,-4) circle (0.65cm);
		\draw (5.6,-4) -- (7.1,-4);
		\node [fill,circle,scale=0.5,label=below:$d_2$] at (7.1,-4) {};
		
		\draw [->] (1,-2.2) -- (1,-0.65) node[midway,left] {$\pi'$};

		\end{tikzpicture}
  \caption{The construction of $D'$ from $D$}
  \label{fig: amalgamation2}
\end{figure}
	Define $h_2=h_2'\circ \pi'$ and $h_1\colon D\to B$ by 
	$$h_1(x)=\begin{cases}
	x & \text{if }x\in B_2 \\
	h'_1(x) & \text{otherwise},
	\end{cases}$$
	where we identify points of $B\setminus B_2$ with their image in $B'$ through $\pi$. Note that $h_1$ is an epimorphism in $\F_P$, indeed it is (weakly) coherent at points of $B_2$ because they are witnesses for their own (weak) coherence and it is (weakly) coherent elsewhere because $h'_1$ is. Moreover we have $h_1'\circ\pi'=\pi\circ h_1$ so that the following diagram commutes:
	$$\begin{tikzcd}
	D \arrow[r, "h_1"] \arrow[d, "\pi'"']& B \arrow[d, "\pi"'] &   \\
	D' \arrow[r, "h_1'"] \arrow[rd, "h_2'"', bend right] & B' \arrow[r, "f'"]                             & A \\
	& C \arrow[ru, "g"', bend right]                 &  
	\end{tikzcd}$$
	It only remains to check that $(f\circ h_1)(x)=(g\circ h_2)(x)$ for all $x\in D$, for which there are two cases: 
	\begin{itemize}
		\item If $x\in B_2$ then $h_1(x)=x$ and $f(h_1(x))=a$, while $h_2(x)=h'_2(\pi'(x))= h_2'(d_1)$, but $g\circ h_2'=f'\circ h_1'$ and $f'(h_1'(d_1))=a$ so $g(h_2(x))=a$. 
		\item Otherwise $h_1(x)$ and $h_2(x)$ agree with $h'_1(x)$ and $h'_2(x)$ respectively, so we immediately obtain $f(h_1(x))=g(h_2(x))$ by construction.
	\end{itemize}
	This concludes the argument when $B_2\neq B_1$. 
	
	\textbf{Case 1b: $\mathbf{B_2=B_1}$.} Let $b\in B$ be the witness for the (weak) coherence of $f$ at $a$ and let $c\in C$ be the witness for the (weak) coherence of $g$ at $a$. Suppose now that for every component $\hat B$ of $B\setminus\{b\}$ that meets $f^{-1}(a)$, we have $\hat B\subseteq f^{-1}(a)$, otherwise modulo choosing a different component $B_1$ of $B\setminus\{b\}$ we would be in the previous case. Similarly we can assume that for every component $\hat C$ of $C\setminus\{c\}$ that meets $g^{-1}(a)$, we have $\hat C\subseteq g^{-1}(a)$, since otherwise we would be in the previous case modulo renaming $B$ and $C$. Let $a_1^A,\ldots,a_n^A$ enumerate the connected components of $A\setminus\{a\}$. By the weak coherence of $f$ we can find $a_1^B,\ldots,a_n^B$ components of $B\setminus\{b\}$ such that $f^{-1}(a_i^A)\subseteq a_i^B$ for all $i$. We define $a_1^C,\ldots,a_n^C$ analogously. Now let $B_1=f^{-1}(a)$ and $C_1=g^{-1}(a)$.
	
	The idea now is to amalgamate the diagram \begin{equation}a_i^B\cup\{b\}\to \{a\}\cup\{a_i^A\}\leftarrow a_i^C\cup\{c\}\label{diagram}\end{equation} for every $i$, then to obtain $D_1$ from $B_1$ and $C_1$ by the joint projection property, and to glue everything back together at the end. 
	By inductive hypothesis we can amalgamate diagrams of the form \eqref{diagram} (note that the trees involved are still members of $\F_P$ and so are the maps) obtaining an amalgam $a_i^D$ with maps $h_1^i\colon a_i^D\to a_i^B\cup\{b\}$ and $h_2^i\colon a_i^D\to a_i^C\cup\{c\}$. Note that, since $b$ and $c$ are the only preimages of $a$ in $a_i^B\cup\{b\}$ and $a_i^C\cup\{c\}$ respectively, there must be an endpoint $x_i\in a^D_i$ such that $h_1^i(x_i)=b$ and $h_2^i(x_i)=c$ for every $i$. 
	
	We now need to deal with $B_1$ and $C_1$. Enumerate as $\widehat{c_1},\ldots,\widehat{c_r}$ the components of $C_1\setminus\{c\}$ and as $\widehat{b_1},\ldots,\widehat{b_s}$ those of $B_1\setminus\{b\}$. Assume without loss of generality $s\leq r$. For $i\leq r$, let $c_i$ be the only vertex of $\widehat{c_i}$ such that $\langle c_i,c\rangle\in E(C)$ and let $b_i$, for $i\leq s$, be defined analogously. Consider now the tree $D_1$ built from $$\left(\bigsqcup_{i\leq r}\widehat{c_i}\right)\sqcup\left(\bigsqcup_{i\leq s}\widehat{b_i}\right)$$ by adding $r+1$ new vertices $z,z_1,\ldots,z_r$ and edges $\langle z_i,c_i\rangle$ for all $i\leq r$, $\langle z_i,b_i\rangle$ for all $i\leq s$ and $\langle z,z_i\rangle$ for all $i\leq r$.

	Finally build a tree $D$ by identifying $z\in D_1$ and $x_i\in a_i^D$ for every $i$, and calling $d$ the image of all those points in $D$, add neighbours to $ z_i, b_i$ and $c_i$ until they have order in $P$. Define a map $h_1\colon D\to B$ by $$h_1(y)=\begin{cases}
	h_1^i(y) & \text{if } y\in a_i^D \\
	y & \text{if } y\in\widehat{b_i}  \\
	b & \text{otherwise} 
	\end{cases}.$$
	Similarly define a map $h_2\colon D\to C$ by 
	$$h_2(y)=\begin{cases}
	h_2^i(y) & \text{if } y\in a_i^D \\
	y & \text{if } y\in \widehat{c_i} \\
	c & \text{otherwise}
	\end{cases}.$$
	It is easy to check that $D$ together with the maps $h_1,h_2$ is an amalgam of $B$ and $C$ over $A$.
	
	This concludes the argument for a ramification point $a\in A$. 
	
	\textbf{Case 2: $\mathbf{a}$ is an endpoint.} Suppose now that $a\in A$ is an endpoint. The construction is essentially an easier version of the one we just carried out for the case in which $B_1\neq B_2$. Since $B\setminus f^{-1}(a)$ is connected we can let $b\in f^{-1}(a)$ be the unique vertex for which there is a $b'\in B\setminus f^{-1}(a)$ with $\langle b,b'\rangle\in E(B)$ and let $B_2$ be $f^{-1}(a)$. As in the previous case consider $\pi\colon B\to B'$ collapsing $B_2$ to $b$ and let $f'\colon B'\to A$ be the unique map with $f'\circ\pi=f$. Moreover let $D'$ be an amalgam of $f'$ and $g$ over $A$ with maps $h_1'\colon D'\to B'$ and $h_2'\colon D'\to C$ such that $g\circ h_2'=f'\circ h_1'$. Once again we build $D$ from $D'$ by glueing $B_2$ to $D'$ in the right spot.  Let $\alpha\in(h_1')^{-1}(b)$ be an endpoint and take the disjoint union of $B_2$ and $D'$, identifying $\alpha\in D'$ and $b\in B_2$. Let $d$ be the image of $\alpha$ and $b$ through this identification.
	
	As in the previous case, we obtain a map $\pi'\colon D\to D'$ by mapping the whole of $B_2$ to $d$, and a map $h_1\colon D\to B$ defined by   
	$$h_1(x)=\begin{cases}
	x & \text{if }x\in B_2 \\
	h'_1(x) & \text{otherwise}.
	\end{cases}$$
	Now by considering $h_2=h_2'\circ\pi'$ we see that $D$ with the maps $h_2\colon D\to C$ and $h_1\colon D\to B$ defined above gives an amalgam of $f$ and $g$ over $A$, as desired.
\end{proof}

As mentioned before Lemma \ref{lemma: amalgamation is satisfied}, the construction in its proof also shows the following:

\begin{lemma}
	The family $\mathcal G_P$ satisfies the projective amalgamation property, so it is a projective \Fraisse family.
\end{lemma}

Now that we know that $\F_P$ and $\mathcal G_P$ are projective \Fraisse families let $\mathbb F_P$ and $\G_P$ denote their projective \Fraisse limits. We know from Lemma \ref{lemma: splitting implies prespace} and Theorem \ref{thm: limit is dendrite} that $|\mathbb F_P|$ and $|\G_P|$ are dendrites, and in the rest of this section we will identify which dendrites they are exactly. Let's start with the easiest case, namely $\omega\not\in P$, so that we're looking at the family $\mathcal G_P$ with projective \Fraisse limit $|\G_P|$.

\begin{thm}\label{thm: limit of G_P}
	Fix $P\subseteq\{3,4,5,\ldots\}$. Then $|\G_P|\cong W_P$.
\end{thm}
\begin{proof}
	For the remainder of this proof we fix a \Fraisse sequence $G_i$ for $\mathcal G_P$, with epimorphisms $f^i_j\colon G_i\to G_j$. We consider $\G_P$ as a subspace of $\prod G_i$, and call $f_i\colon\G_P\to G_i$ the canonical projections. We know that $|\G_P|$ is a dendrite, so we need to check that all ramification points of $|\G_P|$ have order in $P$, and that for all $p\in P$ the set $$\{x\in|\G_P|\mid \ord(x)=p\}$$ is arcwise dense in $|\G_P|$. The first part follows quickly from previous results: if $x\in|\G_P|$ is a ramification point, then $\pi^{-1}(x)=(x_i)$ is a coherent point in $\G_P$ by Theorem \ref{thmw: splitting implies coherent}, where $\pi\colon\G_P\to|\G_P|$ is the topological realization. Note that $\ord(x_i)\in P$ for large enough $i$, so by Remark \ref{rmk: order of weakly coherent} we have that, for $k$ big enough, $\ord(x)=\ord(x_k)\in P$. It remains to show that for every $p\in P$ the set of ramification points in $|\G_P|$ is arcwise dense in $|\G_P|$. Fix $p\in P$ and $X\subseteq|\G|$ an arc. Let $a',b'\in X$ be two distinct points and fix $a\in\pi^{-1}(a'),b\in\pi^{-1}(b')$. Let $i$ be big enough to have $f_i(a)\neq f_i(b)$ and let $\langle w,y\rangle$ be any edge on the unique arc in $G_i$ joining $f_i(a)$ and $f_i(b)$. Consider the graph $H$ obtained from $G_i$ by splitting the edge $\langle w,y\rangle$ into two edges $\langle w,x^i_p\rangle$,$\langle x^i_p,y\rangle$ and by adding $p-2$ new vertices $z_1,\ldots,z_{p-2}$ with edges $\langle x^i_p,z_j\rangle$ for all $1\leq j\leq p$. Note that there is an epimorphism $f^H_i\colon H\to G_i$ in $\mathcal G_P$ with $f^H_i(v)=w$ for all $v\in V(H)\setminus V(G_i)$ and the identity otherwise, and note that in $H$ there is a point of order $p$ on the arc connecting $f_i(a)$ to $f_i(b)$. We can now find $k$ big enough so that there is an epimorphism $f^k_H\colon G_k\to H$ in $\mathcal G_P$ with $f^H_i\circ f^k_H=f^k_i$ and note that in $G_k$ there is a point of order $p$ on the arc connecting $f_k(a)$ and $f_k(b)$, namely the point $x^k_p\in (f^k_H)^{-1}(x^i_p)$ witnessing the coherence of $f^k_H$ at $x^i_p$. This is because $f_k(a)$, $f_k(b)$ are in different components of $G_k\setminus\{x_p^k\}$ by coherence, and $G_k$ is uniquely arcwise connected. Now for all $m>k$ let $x^m_p\in G_m$ be a point witnessing the coherence of $f^m_{m-1}\colon G_m\to G_{m-1}$ at $x^{m-1}_p$ and note that for all $m>k$, $x^m_p$ is a point of order $p$ on the unique arc in $G_m$ joining $f_m(a)$ and $f_m(b)$. Since $x=(x^m_p)_{m\in\Bbb N}\in\G_P$ is a coherent sequence it determines a ramification point of order $p$ in $|\G_P|$ by Theorem \ref{thmw: coherent implies splitting} and Remark \ref{rmk: order of weakly coherent}. By construction $a$ and $b$ belong to different connected components of $\G_P\setminus\{x\}$, and since $\G_P$ is hereditarily unicoherent this implies that any connected set in $\G_P$ containing both $a$ and $b$ must contain also $x$. In particular $x\in\pi^{-1}(X)$ and so $\pi(x)\in X$.
\end{proof}

We can now analyze what happens in the case $\omega\in P$, so that we are looking at the families $\mathcal F_P$ with projective \Fraisse limit $\mathbb F_P$ instead. The situation is more subtle.
\begin{thm}\label{thm: limit of F_P}
	Fix $P\subseteq\{3,4,5,\ldots,\omega\}$. If $\{3,4,5,\ldots,\omega\}\setminus P$ is infinite and $\omega\in P$, then $|\mathbb F_P|\cong W_P$.
\end{thm}
\begin{proof}
	The proof is very similar to that of Theorem \ref{thm: limit of G_P}. For the remainder of this proof we fix a \Fraisse sequence $F_i$ for $\mathcal F_P$, with epimorphisms $f^i_j\colon F_i\to F_j$. We consider $\mathbb F_P$ as a subspace of $\prod F_i$, and call $f_i\colon\mathbb F_P\to F_i$ the canonical projections. We already know that $|\mathbb F_P|$ is a dendrite so we need to check that all ramification points of $|\mathbb F_P|$ have order in $P$, and that for all $p\in P$ the set $$\{x\in|\mathbb F_P|\mid \ord(x)=p\}$$ is arcwise dense in $|\mathbb F_P|$. The first part is very similar to the argument in the proof of Theorem \ref{thm: limit of G_P}: if $x\in|\mathbb F_P|$ is a ramification point, then $\pi^{-1}(x)=(x_i)$ is a weakly coherent point of $\mathbb F_P$ by Theorem \ref{thmw: coherent implies splitting}, where $\pi\colon\mathbb F_P\to|\mathbb F_P|$ is the topological realization. Now the way the family $\F_P$ is constructed allows two options: either $\ord(x_i)$ becomes constant for big enough $i$, so that $(x_i)$ is actually a coherent point of $\mathbb F_P$ and we are in the same case as in the proof of Theorem \ref{thm: limit of G_P}, or $\ord(x_i)\to\infty$ as $i\to\infty$, which by Remark \ref{rmk: order of weakly coherent} implies that $x$ is a ramification point of $|\Bbb F_P|$ of infinite order. In either case $\ord(x)\in P$. 
	
	It remains to show that for every $p\in P$ the set of ramification points in $|\mathbb F_P|$ is arcwise dense in $|\mathbb F_P|$. If $p\neq\omega$ we can repeat verbatim the argument given in the proof of Theorem \ref{thm: limit of G_P}, so we need to show that ramification points of infinite order are arcwise dense in $|\mathbb F_P|$. Fix $X\subseteq|\mathbb F_P|$ an arc, distinct $a',b'\in X$, $a\in\pi^{-1}(a'),b\in\pi^{-1}(b')$ and $i$ big enough to have $f_i(a)\neq f_i(b)$. Fix also a strictly increasing sequence $(n_j)_{j\in\Bbb N}$ of natural numbers, such that $n_j\not\in P$ for all $j$ and $n_j\to\infty$ as $j\to\infty$. Let $\langle w,y\rangle$ be any edge on the unique arc in $F_i$ joining $f_i(a)$ and $f_i(b)$. Consider the graph $H$ obtained from $F_i$ by splitting the edge $\langle w,y\rangle$ into two edges $\langle w,x_1\rangle$,$\langle x_1,y\rangle$ and by adding $n_1-2$ new vertices $z_1,\ldots,z_{n_1-2}$ with edges $\langle x_1,z_j\rangle$ for all $1\leq j\leq n_1-2$. Note that there is an epimorphism $f^H_i\colon H\to F_i$ in $\F_P$ with $f^H_i(v)=w$ for all $v\in V(H)\setminus V(F_i)$ and the identity otherwise, and note that in $H$ there is a point of order exactly $n_1$ on the arc connecting $f_i(a)$ to $f_i(b)$. We can now find $k$ big enough so that there is an epimorphism $f^k_H\colon F_k\to H$ in $\F_P$ with $f^H_i\circ f^k_H=f^k_i$ and note that in $F_k$ there is a point of order at least $n_1$ on the arc connecting $f_k(a)$ and $f_k(b)$, namely the point $x_2\in (f^k_H)^{-1}(x_1)$ witnessing the weak coherence of $f^k_H$ at $x_1$. Now we can iterate this construction, by considering the tree $H_2$ obtained from $F_k$ by adding at least $n_s-n_1-2$ neighbours to $x_2$, with the map $f\colon H_2\to F_k$ sending all of those new vertices to $x_2$, where $s$ is big enough to have $\ord(x_2)<n_s-n_1-2$. Then we can find some $l>k$ and an epimorphism $f^l_H\colon F_l\to F_H$ such that $f^{H_2}_k\circ f^l_{H_2}=f^l_k$ and as above we can find a point $x_3\in F_l$ of order at least $n_s$ on the arc between $f_l(a)$ and $f_l(b)$. Repeating this construction we build a weakly coherent sequence $(x_i)$ such that $\ord(x_i)\to\infty$ and $x_i$ lies on the unique arc in $F_i$ between $f_i(a)$ and $f_i(b)$. The same argument used at the end of the proof of Theorem \ref{thm: limit of G_P} allows us to conclude that the ramification point of infinite order $\pi(x_i)\in|\mathbb F_P|$ belongs to the arc $X$.
\end{proof}

\begin{remark}
	In Theorem \ref{thm: limit of F_P} we had to assume that $\{3,4,5,\ldots,\omega\}\setminus P$ is infinite, which was used to find a sequence $n_j$ of natural numbers with $\lim n_j\to\infty$ and $n_j\not\in P$, in order to construct weakly coherent points of $\mathbb F_P$ corresponding to ramification points of $|\mathbb F_P|$ of infinite order. When $P$ is cofinite in $\{3,4,5,\ldots,\omega\}$ the family $\mathcal F_P$ is still a projective \Fraisse family, as we showed above, but the problem is that the order of the points in weakly coherent sequences must stabilize instead of growing unboundedly. Indeed the techniques used in the proof of Theorem \ref{thm: limit of F_P} easily establish that $|\mathbb F_P|=W_{P'}$ where $P'=P\setminus\{\omega\}$ if $P=\{3,4,5,\ldots,\omega\}$ and $P'=(P\setminus\{\omega\})\cup \{a\}$ otherwise, where $a$ is the smallest integer such that for all $n>a$, $n\in P$.
\end{remark}

\section{\Fraisse categories and Projection-Embedding pairs}
In this section, we construct all generalized \Wazewski dendrites as projective \Fraisse limits by moving to the more general setting of \Fraisse categories developed by \Kubis in \cite{K}. Those are categories in which the amalgamation and joint embedding properties hold (when expressed in terms of the appropriate diagrams). As an application we recover a countable dense homogeneity result for $\End(W_P)$. We quickly recall some definitions from \cite{K}, but stress that the approach followed there is far more general than what we need. In particular, since we are only working with countable collections of finite objects (with finitely many morphisms between any two of them) we can ignore all the issues regarding existence of \Fraisse sequences and uniqueness of \Fraisse limits that arise when looking at the uncountable case. Of particular interest to us is section 6 of \cite{K} in which projection-embedding pairs are introduced as a tool to produce objects that are universal both in an injective and a projective sense.

\subsection{\Fraisse Categories}\hfill\\
The following definitions and results are taken from \cite[Section 2-3]{K}, where they are presented in a more general way, which works in the uncountable case as well.

As a warning let us point out that while \Kubis's approach is phrased in terms of injective \Fraisse limits we decided to phrase everything in terms of projective \Fraisse limits. There is no real difference between the two approaches, since it is enough to look at the opposite category $\K^{\mathrm{op}}$ to translate between the two.

Given a category $\K$ and two objects $a,b\in\K$, let $\K(a,b)$ denote the morphisms $a\to b$ in $\K$.

\begin{defn}
	Let $\K$ be a category. We say that $\K$ has the \emph{amalgamation property} if for all $a,b,c\in\K$ and all morphisms $f\in\K(b,a),g\in\K(c,a)$, there exist $d\in\K$ and morphisms $f'\in\K(d,b),g'\in\K(d,c)$ with $f\circ f'=g\circ g'$. We say that $\K$ is \emph{directed} or that $\K$ has the \emph{joint projection property} if for all $a,b\in\K$ there exists $c\in\K$ such that $\K(c,a)\neq\varnothing\neq\K(c,b)$.
\end{defn}

We now wish to define a category of projective sequences in $\K$, which we denote by $\sigma\K$, whose morphisms will be defined in a way that guarantees the following property. Suppose that $\bar{x},\bar{y}\in\sigma\K$ and $\K$ is embedded in a category in which those two sequences have a limit. Given an arrow in $\sigma\K(\bar{x},\bar{y})$ we want an induced arrow $\varprojlim \bar{x}\to\varprojlim\bar{y}$. Given $\bar{x}\in\sigma\K$ we denote $x(n)$ by $x_n$ and the morphism $x(n)\to x(m)$ for $n\geq m$ by $x^n_m$. By definition a sequence in $\K$ is a functor $\omega\to\K$, where we think about the former as a poset category with the reverse order, so a transformation $\bar{x}\to\bar{y}$ is by definition a natural transformation $F\colon\bar{x}\to\bar{y}\circ\varphi$, where $\varphi\colon\omega\to\omega$ is order preserving. To define 
an arrow $\bar{x}\to\bar{y}$ in $\sigma\K$ we need to identify some of those natural transformations.

\begin{defn}
	Let $F\colon\bar{x}\to\bar{y}\circ\varphi$ and $G\colon\bar{x}\to\bar{y}\circ\psi$ be two natural transformations. We say that $F$ and $G$ are equivalent if
	\begin{enumerate}
		\item For every $n$ there exist $m\geq n$ such that $\varphi(m)\geq\psi(n)$ and $$y_{\psi(n)}^{\varphi(m)}\circ F(m)=G(n)\circ x^m_n;$$
		\item For every $n$ there exist $m\geq n$ such that $\psi(m)\geq\varphi(n)$ and $$y^{\psi(m)}_{\varphi(n)}\circ G(m)=F(n)\circ x^m_n.$$
	\end{enumerate}
An arrow $\bar{x}\to\bar{y}$ in $\sigma\K$ is an equivalence class of natural transformations $\bar{x}\to\bar{y}$.
\end{defn}

Note that the category $\sigma\K$ in \Kubis's approach is the analogue of the class of structures $\mathcal F^\omega$ in the approach by Panagiotopolous and Solecki. We can now define a \Fraisse sequence for a category $\K$, which is the analogue of the \Fraisse sequence in the classical or projective \Fraisse limit construction. \Kubis defined in an analogous manner a category of sequences of length $\kappa$ for any cardinal $\kappa$. Similarly the following definition is a special case of the definition of a $\kappa$-\Fraisse sequence from \cite{K}.

\begin{defn}
	Let $\K$ be a category. A \Fraisse sequence in $\K$ is a projective sequence $\bar{u}\colon \omega\to\K$ satisfying:
	\begin{enumerate}
		\item For every $x\in\K$ there exist $n<\omega$ with $\K(u_n,x)\neq\varnothing$.
		\item For every $n<\omega$ and every arrow $f\in\K(y,u_n)$, with $y\in\K$, there exist $m\geq n$ and $g\in\K(u_m,y)$ such that $f\circ g=u^m_n$.
	\end{enumerate}
	Note that if $\K$ has the amalgamation property then $\bar{u}$ also satisfies
	\begin{enumerate}\setcounter{enumi}{2}
		\item For all arrows $f\in\K(a,b)$, $g\in\K(u_n,b)$, there exist $m\geq n$ and $h\in\K(u_m,a)$ with $g\circ u^m_n=f\circ h$.
	\end{enumerate}
\end{defn}

The following is a special case of Corollary 3.8 of \cite{K} and shows existence of \Fraisse sequences:
\begin{lemma} \label{lemma: fraisse sequence exists}
	Let $\K$ be a category with the amalgamation and joint embedding property with countably many objects and such that between any two objects there are finitely many arrows. Then $\K$ has a \Fraisse sequence.
\end{lemma}
The following is Theorem 3.15 of \cite{K} and shows the uniqueness (up to isomorphism) of \Fraisse sequences, hence also the uniqueness of \Fraisse limits when we embed into a category in which the \Fraisse sequence has a limit:
\begin{thm}\label{thm: fraisse limit is unique}
	Assume $\overline{u}$ and $\overline{v}$ are \Fraisse sequences in a given category $\K$. Assume further that $k,l<\omega$ and $f\in\K(u_k,v_l)$. Then there exists an isomorphism $F\colon\overline{u}\to\overline{v}$ in $\sigma\K$ such that the diagram 
	\begin{center}
		\begin{tikzcd}
	 	\overline{u} \arrow[d,"u_k"']\arrow[r, "F"] & \overline{v} \arrow[d,"v_l"'] \\
	 	u_k \arrow[r,"f"'] & v_l
		\end{tikzcd}
	\end{center}
	commutes. In particular $\overline{u}\approx\overline{v}$.
\end{thm}

\subsection{Generalized \Wazewski Dendrites from \Fraisse Categories}\hfill
We now have all the tools needed to talk about projection-embedding categories and construct the generalized \Wazewski dendrites not obtained in previous sections.
\begin{defn}\label{def: dF}
	Given $P\subseteq\{3,4,\ldots,\omega\}$ let $\L_P'$ denote the language with a binary relation $R$ and unary relations $U_p$ for every $p\in P$. Let $\mathcal L_R=\{R\}\subseteq\mathcal L_P'$ and let $\mathcal L_P=\{U_p\mid p\in P\}\subseteq\mathcal L_P'$. Let $\dF$ denote the following category:
	\begin{itemize}
		\item $A\in\dF$ iff $A$ is a finite tree (with $R$ as the edge relation), $A$ contains at least one ramification point and every vertex of $A$ is either an endpoint or a ramification point. Moreover for every ramification point $a\in A$ there is exactly one $p\in P$ such that $U_p(a)$ holds, and in that case $\ord(a)\leq p$. When $a$ is a ramification point we will refer to the unique $p\in P$ for which $U_p(a)$ holds as the \emph{label} of $a$. If $a$ instead is an endpoint of $A$, then $U_p(a)$ doesn't hold for any $p\in P$.
		\item Given $A,B\in\dF$ an arrow $f\in\dF(B,A)$ is a pair $(p(f),e(f))$ where $p(f)\colon B\to A$ is a weakly coherent epimorphism from $B$ to $A$ viewed as $\mathcal L_R$-structures with the discrete topology. Moreover if for $a\in A$, $U_p(a)$ holds and $b\in B$ witnesses the weak coherence of $p(f)$ at $a$, then $U_p(b)$ holds. In particular, we are not requiring that if $U_p(b)$ holds, then $U_p(p(f)(b))$ also holds. On the other hand $e(f)\colon \End(A)\to \End(B)$ is a map associating to every $a\in \End(A)$ a $b\in\End(B)$ with $p(f)\circ e(f)=\Id_{\End(A)}$, where $\End(X)$ is the set of endpoints of $X$. In other words $p(f)$ is a weakly coherent epimorphism between the $\mathcal L_R$-reducts, while $e(f)$ is a partial right inverse defined on the endpoints.
	\end{itemize} 
	Given two arrows $f\colon B\to A$ and $g\colon C\to B$ in $\dF$ we define their composition in the obvious way: $f\circ g$ is the pair $(p(f)\circ p(g),e(g)\circ e(f))$.
	It is clear that composition is associative and that the pair $(\mathrm{Id}_A,\mathrm{Id}_A)$ is the identity morphism for $A$ in $\dF$, so $\dF$ is a category.
\end{defn}

\begin{remark}
	The argument used to prove Lemma \ref{lemma: joint projection property} shows that $\dF$ is directed, so in order to check that it is a \Fraisse category we only need to check that it has the amalgamation property.
\end{remark}

\begin{lemma}\label{lemma: amalgamation for EP pairs}
	The category $\dF$ has amalgamation. In other words whenever we have arrows $f\in\dF(B,A)$, $g\in\dF(C,A)$, there exists $D\in\dF$ with arrows $h_1\in\dF(D,B)$, $h_2\in\dF(D,C)$ such that $f\circ h_1=g\circ h_2$.
\end{lemma}
\begin{proof}
The proof is very similar to the proof of amalgamation given in the previous section for $\mathcal F_P$, so we don't write all the details. Instead we point out the extra steps that need to be taken to adapt the proof of \ref{lemma: amalgamation is satisfied} to the current setting. Let $A,B,C,f,g$ be as in the statement of the lemma. We proceed by induction on $|E(B)|+|E(C)|$. The smallest structures in $\dF$ are trees with a single ramification point labelled with any $p\in P$ that has $3$ neighbours. Clearly if $B$ and $C$ are of that form then $A$ must be the same tree on $4$ vertices as well, and all the ramification points in $A,B,C$ must have the same label, so we can take $D=A=B=C$ with identity maps to be the amalgam.
Our induction hypothesis now becomes that not only diagrams of the form $B\to A\leftarrow C$ with smaller $|E(B)|+|E(C)|$ can be amalgamated in $\dF$, but that the amalgamation can be carried on in a way that preserves labels and is compatible with the embedding part of the arrows. With this assumption we can use the same construction as in the proof of Lemma \ref{lemma: amalgamation is satisfied}. Indeed as in that lemma suppose without loss of generality that $p(f)^{-1}(a)$ is nontrivial for some $a\in A$. Suppose first that $a$ is a ramification point and we are in the $B_2\neq B_1$ case. We can produce $B'$ with $\pi\in\dF(B,B')$, $f'\in\dF(B',A)$ and $D'$ with $h_1'\in\dF(D',B')$ and $h_2'\in\dF(D',C)$ as in the proof of Lemma \ref{lemma: amalgamation is satisfied}. Note that in this case $p(\pi)$ is the identity map between $\End(B)$ and $\End(B')$, so that $e(f')$ can be defined to agree with $e(f)$. We then build $D$ from $D'$ by glueing back $B_2$ along an edge. Points in $D$ are already labeled correctly, with labels coming from either $B_2$ or $D'$, and since glueing $B_2$ into $D'$ doesn't collapse any endpoints in $D'$ we can take $e(h_2)=e(h_2')$, while $e(h_1)=e(h_1')$ on $\End(B')$ and is the identity on $\End(B_2)\setminus\{b_1,b_2\}$. 

The $B_2=B_1$ is easier to deal with: we follow the same construction as in the proof of Lemma \ref{lemma: amalgamation is satisfied} and we never run into issues with the embedding part of the morphisms, since endpoints are never collapsed during the construction.

If instead $a$ is an endpoint there are still no issues mimicking the construction from Lemma \ref{lemma: amalgamation is satisfied} as far as labels are concerned, but a little more care is needed for the embedding part of the arrows, so we will write out some of the details. Let $b\in p(f)^{-1}(a)$ be the unique vertex for which there is a $b'\in B\setminus p(f)^{-1}(a)$ with $\langle b',b\rangle\in E(B)$ and let $B_2$ be $p(f)^{-1}(a)$. As in the previous case we construct $B'$ from $B$ by collapsing $B_2$ to $b$ through the map $p(\pi)\colon B\to B'$ but now we also need to define the embedding part $e(\pi)$ in order to get a morphism in $\dF$. Let $e(\pi)$ be the identity on $\End(B')\setminus\{b\}$ and $e(\pi)(p(\pi)(b))=e(f)(a)$. Similarly we obtain a coherent epimorphism $p(f')\colon B'\to A$ as the only such morphism satisfying $p(f')\circ p(\pi)=p(f)$. There is only one choice for $e(f')$, namely
$$e(f')(x)=\begin{cases}e(f)(x) & \text{ if $x\in\End(A)\setminus\{a\}$} \\ p(\pi)(b) & \text{ if $x=a$}\end{cases}.$$
Once again we obtain, by inductive hypothesis, an amalgam $D'$ with $h_1'\in\dF(D',B')$ and $h_2'\in\dF(D',C)$ such that $f'\circ h_1'=g\circ h_2'$. We now have to glue back $B_2$ to $D'$ in order to build an amalgam $D$ with maps $h_1\in\dF(D,B)$ and $h_2\in\dF(D,C)$. We already know how to build $D$ as a graph and how to define $p(h_1),p(h_2)$, since the construction is  the same as the one in the proof of Lemma \ref{lemma: amalgamation is satisfied}. As in the previous section we obtain $D$ from $D'$ and $B_2$ by identifying $b\in B_2$ with an endpoint of $D'$. Here we take $\alpha\in\End(D')$ to be $e(h_1')(b)$ (which is also equal to $e(h_2')(e(g)(a))$). Points of $D$ are already correctly labeled, with labels coming from either $B_2$ or $D$, and we define $p(h_1),p(h_2)$ exactly as above. We only need to define $e(h_1),e(h_2)$. There is nothing to be done for $e(h_2)$ which we simply define to be equal to $e(h_2')$ on all endpoints except $e(g)(a)\in B_2$, where it now takes value $e(f)(a)$. To define $e(h_1)$ there are a few cases to consider, as follows: 
$$e(h_1)(x)=\begin{cases}
	e(h_1')(x) & \text{if } x\in\End(B)\setminus\End(B_2) \\
	x & \text{if } x\in\End(B_2)
\end{cases}.$$ 

It is now easy to check that the resulting square diagram is commutative, as desired.
\end{proof}

Since $\dF$ has countably many objects and between any two objects of $\dF$ there are finitely many arrows, Lemma \ref{lemma: fraisse sequence exists} and Theorem \ref{thm: fraisse limit is unique} immediately imply the following:
\begin{thm}
	The category $\dF$ has a \Fraisse sequence $\overline{u}$ and any two \Fraisse sequences are isomorphic as elements of $\sdF$. In particular whenever $\dF$ is embedded in a category in which \Fraisse sequences have colimits, the colimits arising from two \Fraisse sequences are isomorphic.
\end{thm}

As mentioned in the introduction to this section we want to embed $\sdF$ in a category in which those sequences actually have a limit. We do so by considering $\sdF$ as a subcategory of a category $\mathcal G$ of \textit{labeled topological graphs} defined as follows: The objects of $\mathcal G$ are pairs $G=(B,E)$ where $B$ is a graph-dendrite, equipped with unary predicates $U_n$, for $n\in P$, such that if $a$ is a ramification point of $G$, there is a unique $n$ for which $U_n(a)$ holds and $n\leq\ord(a)$, and $E$ is a countable subset of $\Bbb U$. The morphisms of $\mathcal G$ from $(B,E)$ to $(A,F)$ are pairs $f=(p(f),e(f))$ such that $e(f)\colon F\to E$ is an injection, while $p(f)\colon B\to A$ is an epimorphism of topological graphs such that $p(f)\circ e(f)=\mathrm{Id}_F$. 

To a sequence $\overline{u}\in\sdF$ we associate the pair $(\Bbb U,E)$, where $\Bbb U$ is the topological graph obtained as the projective limit of the projection part of $\overline{u}$, while $E$ is the direct limit of the embedding part of $\overline{u}$. Moreover every ramification point $x$ of $\Bbb U$ is labeled with its order as label, where the order of $x$ is the (potentially infinite) number of components of $\Bbb U\setminus\{x\}$. Given another $\overline{v}\in\sdF$, a morphism $\psi\colon \overline{u}\to\overline{v}$ is by definition a natural transformation $\psi$ from $\overline{u}$ to $\overline{v}\circ\varphi$, where $\varphi\colon\omega\to\omega$ is order-preserving. It is easy to check that if $(\Bbb U,E)$ and $(\Bbb V,F)$ are the objects associated to $\overline{u}$ and $\overline{v}$ as described above, then the projection part of $\psi$ converges to a continuous surjection $p(\psi)\colon\Bbb U\to\Bbb V$, while the embedding part converges to an injection $e(\psi)\colon F\to E$ such that $p(\psi)\circ e(\psi)=\mathrm{Id}_F$. Note that, since in Definition \ref{def: dF} we required that the witness of weak coherence at a point $a$ has the same label as $a$, $p(\psi)$ has the property that if $v\in\Bbb V$ is a ramification point with label $l$, then there is a ramification point $u\in p(\psi)^{-1}(v)$ with the same label.
This describes the desired functor $\sdF\to\mathcal G$. 

To a sequence $\overline{u}\in\sdF$ we associated a pair $(\Bbb U,E)$, where $\Bbb U$ is a topological graph. We will show that whenever $\overline{u}$ is a \Fraisse sequence for $\dF$, then $|\Bbb U|$ is the generalized \Wazewski dendrite $W_P$, where $\pi\colon\Bbb U\to|\Bbb U|$ is the topological realization.

Note that the embedding part of $\overline{u}$ plays no role in the following theorem. This observation will be made precise and become relevant later.

\begin{thm} \label{thm: limit of dF is W_P}
	Let $\overline{u}=\langle F_i\mid i<\omega\rangle$ with maps $f^{i+1}_i\in\dF(F_{i+1},F_i)$ be a \Fraisse sequence for $\dF$ and let $(\mathbb F_P,E)$ be its limit. Then $|\Bbb F_P|=W_P$.
\end{thm}
\begin{proof}
	Even though $\dF$ doesn't allow splitting edges, we can proceed as in Remark \ref{remark: splitting edges} to use all the results that required splitting edges. Because of this we know that $|\Bbb F_P|$ is a dendrite thanks to Theorem \ref{thm: limit is dendrite} since $p(f^{i+1}_i)$ is monotone and (weakly) coherent, moreover we still have the correspondence between (weakly) coherent points of $\Bbb F_P$ and ramification points of $|\Bbb F_P|$ established in Theorem \ref{thmw: coherent implies splitting} and Theorem \ref{thmw: splitting implies coherent}. The fact that each ramification point has order in $P$ follows from an argument similar to the one done in the proof of Theorem \ref{thm: limit of F_P}: if $x=(x_n)\in\Bbb F_P$ is labelled with $U_p$, then $x_n$ must be labelled with $U_p$ from some index on. If $p<\omega$, then $\ord(x_n)$ must also be $p$ from some (possibly bigger) index on, which implies that $\ord(x)=p$, and so $\ord(\pi(x))=p$. If $x=(x_n)\in\Bbb F_P$ is labelled with $U_\omega$ instead, then $\ord(x_n)\to\infty$ as $n\to\infty$ as in the proof of Theorem \ref{thm: limit of F_P}, hence $\ord(\pi(x))=\omega$. The fact that the ramification points of each order in $P$ are arcwise dense in $|\Bbb F_p|$ follows from exactly the same argument as the one at the end of the proof of Theorem \ref{thm: limit of F_P}.
\end{proof}

\subsection{\texorpdfstring{A Countable Dense Homogeneity Result for $\mathbf{End(W_P)}$}{A countable Dense Homogeneity Result for WP}}\hfill\\
We now want to prove the following homogeneity result for the endpoints of $W_P$, first stated in \cite{CD} for $W_3$:

\begin{thm} \label{thm: CDH}
	Let $P\subseteq\{3,4,\ldots,\omega\}$ and let $Q_1,Q_2$ be countable dense subsets of $\End(W_P)$. Then there is a homeomorphism $h\colon W_P\to W_P$ with $h(Q_1)=Q_2$.
\end{thm}

We begin with a more precise description of the endpoints of a projective \Fraisse limit of trees, in particular we want to obtain that if $(\Bbb U,E)$ is the projective \Fraisse limit of $\dF$, then $\pi(E)$ is a dense set of endpoints in $|\Bbb U|$. We first show that endpoints are edge related only to themselves.

\begin{prop}\label{prop: endpoints are alone}
	Let $\mathcal F$ be a projective \Fraisse family of trees with (weakly) coherent epimorphisms that allows splitting edges and let $\Bbb F$ be its projective \Fraisse limit. If $e\in\End(\Bbb F)$, then there is no $e'\in\Bbb F\setminus\{e\}$ with $\langle e,e'\rangle\in E(\Bbb F)$.
\end{prop}
\begin{proof}
	Suppose for a contradiction that there is $e\neq e'\in\Bbb F$ with $\langle e,e'\rangle\in E(\Bbb F)$. Let $F_n$ with maps $f^m_n\colon F_m\to F_n$ be a \Fraisse sequence for $\mathcal F$ and let $g_n\colon \Bbb F\to F_n$ be the canonical projections.
	Let $e_n,e'_n$ be $g_n(e)$ and $g_n(e')$ respectively. First we claim that for all but finitely many $n$, $e_n\in\End(F_n)$. Indeed suppose for a contradiction that this is not the case and let $m$ be big enough so that $e_m\neq e'_m$ and $e_m$ is not in $\End(F_m)$. Let $a_m$ be a point in $F_m$ such that $e_m$ is on the unique arc joining $a_m$ and $e'_m$. Without loss of generality we can assume that $a_m$ is a ramification point (we can for example take $a_m=e_m$) and let, for $k>m$, $a_k$ denote the vertex of $F_k$ witnessing the (weak) coherence of $f^k_m$ at $a_m$. By monotonicity of $f^k_m$ and since $\langle e_k,e'_k\rangle\in E(F_k)$, we have that for all $k>m$, $e_k$ is on the unique arc joining $a_k$ and $e'_k$. Indeed $(f^k_m)^{-1}([a_m,e'_m])$ is connected by monotonicity of $f^k_m$ and, since $F_k$ is a tree, there is a unique edge between $(f^k_m)^{-1}([a_m,e_m])$ and $(f^k_m)^{-1}([e_m,e'_m])$ which must be the edge between $e_k$ and $e_k'$. Since the unique arc from $a_k$ to $e'_k$ goes from $(f^k_m)^{-1}([a_m,e_m])$ to $(f^k_m)^{-1}([e_m,e'_m])$ it must go through this edge, showing that $e_k$ lies on the arc from $a_k$ to $e'_k$. As in the proof of arcwise density in Theorem \ref{thm: limit of G_P} we obtain a ramification point $a\in\Bbb F$ such that $e$ is contained in the arc $[a,e']$. Since $e\neq a$, as the latter is a ramification point, and $e\neq e'$ by assumption, this contradicts the fact that $e$ is an endpoint of $\Bbb F$. 

    We have thus obtained that, if $e\in\End(\mathbb F)$ and $e'\in\mathbb F$ is such that $\langle e,e'\rangle\in E(\mathbb F)$, then for all but finitely many $n\in\mathbb N$,
    $e_n$ is an endpoint of $F_n$.
    
    Now let $H$ be the tree obtained by splitting the $\langle e_m,e'_m\rangle$ edge, so that $V(H)=V(F_m)\sqcup\{x\}$ and $$E(H)=(E(F_m)\setminus\{\langle e_m,e'_m\rangle\})\cup\{\langle e_m,x\rangle,\langle x,e'_m\rangle\}$$
	and consider the epimorphism $\varphi\colon H\to F_m$ defined by $\varphi(x)=e_m$ and $\varphi$ is the identity on $H$ otherwise (when applying this result to the $\mathcal F=\dF$ case we will also need to add neighbours to $x$ until it has the correct order). By assumption we can find $n>m$ and an epimorphism $\psi\colon F_n\to H$ such that $\varphi\circ\psi=f^n_m$. Since $\psi$ respects the edge relation we must have $\psi(e_n)=x,\psi(e'_n)=e'_m$. But now, since $e_n\in\End(F_n)$ and $\psi$ is monotone, it is impossible for $e_m$ to be in the image of $\psi$, contradicting that the latter is an epimorphism.
\end{proof}

\begin{lemma} \label{lemma: endpoints of the limit}
	Let $\F$ be a projective \Fraisse family of trees with (weakly) coherent epimorphisms that allows splitting edges. Let $\langle F_i\mid i<\omega\rangle$ be a \Fraisse sequence for $\F$ and let $\Bbb F\subseteq\prod F_i$ be its projective \Fraisse limit. Let $e=(e_n)_{n\in\Bbb N}\in\Bbb F$. If there is $N\in\Bbb N$ such that for every $m\geq N$, $e_m$ is an endpoint of $F_m$, then $e\in\End(\Bbb F)$.
\end{lemma}
\begin{proof}
	Assume that $(e_n)_{n\in\Bbb N}$ is such that $e_m$ is an endpoint of $F_m$ for all $m\geq N$. Suppose by contradiction that $e$ is not an endpoint of $\Bbb F$ and let $A\subseteq\Bbb F$ be an arc with $e\in A$, but $e\neq a_1,a_2$, where $a_1,a_2$ are the endpoints of $A$. Let $k$ be big enough so that $e_k$ is an endpoint of $F_k$, and $f_k(a_1),\,f_k(a_2),\,f_k(e)$ are pairwise distinct. By Theorem \ref{lemma: maps from the limit are monotone} $f_k$ is monotone, so by Lemma \ref{thm: monotone maps preserve endpoints} $f_k(a_1),\,f_k(a_2)$ and $e_k$ all belong to $\End(f_k(A))$, a contradiction.
\end{proof}

\begin{lemma}
		Let $\dF$ be as above, let $\overline{u}=\langle F_i\mid i<\omega\rangle$ be a \Fraisse sequence for $\dF$ with maps $f^m_n\in\dF(F_m,F_n)$ and let $(\Bbb F,E)$ be its projective \Fraisse limit. Then $E\subseteq\End(\Bbb F)$ and $E$ is dense in $\Bbb F$.
\end{lemma}
\begin{proof}
	It follows from Lemma \ref{lemma: endpoints of the limit}, that $E\subseteq\End(\Bbb F)$. Let now $U\subseteq\Bbb F$ be a nonempty open set. We want to show that $E\cap U\neq\varnothing$. Let $n\in\omega$ be sufficiently big so that there is $x\in F_n$ such that $f_n^{-1}(x)\subseteq U$. Consider any finite tree $T\in\dF$ with an epimorphism $f\in\dF(T,F_n)$ such that $p(f)^{-1}(x)$ contains an endpoint of $T$. By definition of a \Fraisse sequence we can find $m\geq n$ and $g\in\dF(F_m,T)$ such that $f\circ g=f^m_n$. Since the preimage of an endpoint through a (weakly) coherent map always contains an endpoint, this implies that there is an endpoint $y\in\End(F_m)$ with $f_m^{-1}(y)\subseteq U$. Now the point of the limit determined by the sequence $e(f^k_m)(y)$ for $k\geq m$ is an endpoint by Lemma \ref{lemma: endpoints of the limit} and is contained in both $E$ and $U$ by construction.
\end{proof}

	As a consequence of the previous two lemmas and Proposition \ref{prop: endpoints are alone}, we obtain the following corollary:
\begin{cor}
If $(\Bbb U,E)$ is the projective \Fraisse limit of $\dF$, then $\pi(E)$ is a dense set of endpoints in $|\Bbb U|\simeq W_P$, where $\pi\colon\Bbb U\to|\Bbb U|$ is the topological realization.
\end{cor}

The last piece we need before being able to prove the desired countable dense homogeneity result for the endpoints of $W_P$ is Theorem \ref{thm: sequence is Fraisse if limit is WP-prespace}. Intuitively this theorem will say that the embedding parts of the morphisms is not important in determining whether a sequence is \Fraisse for $\dF$. This is not surprising, since in Theorem \ref{thm: limit of dF is W_P} the embedding part of the morphisms played no role. We first need to introduce new definitions and make a few simple remarks.

\begin{defn}\label{def: center closed stuff}
	Let $X$ be a dendrite. Given any three distinct $x,y,z\in X$, their \emph{center} $C(x,y,z)$ is the unique point in $[x,y]\cap[y,z]\cap[x,z]$. A set $F\subseteq X$ is called \emph{center-closed} if $C(f_1,f_2,f_3)\in F$ whenever $f_1,f_2,f_3\in F$ are distinct. Given a finite set $A\subseteq X$, its \emph{center closure} is the smallest center-closed $B\subseteq X$ with $A\subseteq B$. Given a finite center-closed $F\subseteq X$ let, for $a\in F$, $\hat{a}_F$ denote the set of connected components of $X\setminus\{a\}$ that contain no point of $F$. Given distinct $a,b\in F$ let $C_a(b)$ be the connected component of $X\setminus\{a\}$ containing $b$, $C_b(a)$ the connected component of $X\setminus\{b\}$ containing $a$, and $C_{a,b}=C_b(a)\cap C_a(b)$. The \emph{partition associated to $F$} is then $$\Omega_F=\bigcup_{a\in F}\hat{a}_F\cup\{C_{a,b}\mid a\neq b\in F \text{ with } [a,b]\cap F=\{a,b\}\},$$
	which is exactly the set of connected components of $X\setminus F$.

    If $X$ is a graph-dendrite instead we have analogous notions, for any three distinct points $x,y,z\in X$ we define their \emph{center} $C(x,y,z)$ to be the unique point in $[x,y]\cap[y,z]\cap[x,z]$, where those are now arcs in the sense of topological graphs. The other notions are defined similarly.
\end{defn}

\begin{defn}
	An \emph{immersion} of a finite tree $A$ in a graph-dendrite $B$ is a map $f\colon A\to B$ (which will usually not be a graph homomorphism) such that for all distinct $a,b,c\in A$, $$b\in[a,c]\iff f(b)\in[f(a),f(c)].$$ In other words there exist a topological graph $A'$ obtained from $A$ by replacing every edge of $A$ with an arc (as in Definition \ref{defn: arcs and more}) and a graph embedding $f'\colon A'\to B$ with $f=\restr{f'}{A}$.
\end{defn}

	Note that the image of an immersion $f\colon A\to B$ is center-closed, so it determines a partition $\Omega_f=\Omega_{f(A)}$ of $B\setminus f(A)$ as above. There are two particular instances of immersions that will be relevant. They are described in the following examples.
	
\begin{example}\label{remark: arrows produce immersions}
	An arrow $f\in\dF(B,A)$ determines an immersion $i(f)\colon A\to B$ by setting
	$$i(f)(a)=\begin{cases}
		e(f)(a) & a\in\End(A) \\
		w(a) & \text{otherwise},
	\end{cases}$$
	where $w(a)\in B$ is the witness for the (weak) coherence of $p(f)$ at $a$.
\end{example}
\begin{example}\label{remark: Immersion in limit of sequence}
	Let $\langle F_n\mid n<\omega\rangle$ with maps $f^m_n\in\dF(F_m,F_n)$ be a sequence in $\dF$ and let $\Bbb F$ with maps $f^\infty_n\in\dF(\Bbb F,F_n)$ be its limit. We have immersions $h_n\colon F_n\to \Bbb F$ given by $$x\mapsto (p(f^n_0)(x),p(f^n_1)(x),\ldots,p(f^n_{n-1})(x),x,w^n_{n+1}(x),w^n_{n+2}(x),\ldots),$$ where $w^i_j(x)$ denotes the witness for the (weak) coherence of $f^j_i\colon F_j\to F_i$ at $x$ for $x$ a ramification point. For endpoints we follow the embeddings instead of the witnesses, that is $$x\mapsto (p(f^n_0)(x),p(f^n_1)(x),\ldots,p(f^n_{n-1})(x),x,e(f^{n+1}_n)(x),e(f^{n+2}_n)(x),\ldots),$$ when $x$ is an endpoint.
\end{example}

\begin{remark}\label{remark: from immersion to coherence}
	Let $i\colon A\to B$ be an immersion of finite trees. If $f\colon B\to A$ is a monotone map such that 
	\begin{itemize}
		\item for every $a\in A$, $f(i(a))=a$,
		\item for every $a\in A$ and every $b\in\bigcup \widehat{i(a)}_{i(A)}$, $f(b)=a$,
		\item for every $a\neq a'\in A$ with $(a,a')\in E(A)$ and every $b\in C_{i(a),i(a')}$, $f(b)\in\{a,a'\}$.  
	\end{itemize}
	Then $f$ is weakly coherent.
\end{remark}

\begin{defn}
We say that a topological graph $G$ is a \emph{$W_P$-prespace} if $G$ is a prespace, $G$ is a graph-dendrite, for every $p\in P$ the set of ramification points of order $p$ is arcwise dense in $G$ and every ramification point of $G$ has order in $P$. In particular if $G$ is a $W_P$-prespace, then $|G|=W_P$.
\end{defn}

We can now characterize \Fraisse sequences in $\dF$.

\begin{thm}\label{thm: sequence is Fraisse if limit is WP-prespace}
	A sequence $\langle F_n\mid n<\omega\rangle$ with maps $f^m_n\in\dF(F_m,F_n)$ is \Fraisse if and only if the inverse limit of $p(f^m_n)$ as a topological graph is a $W_P$-prespace.
\end{thm}
\begin{proof}
	The if part of the statement was already proved in Theorem \ref{thm: limit of dF is W_P}. Suppose that $\langle F_n\mid n<\omega\rangle$ with epimorphisms $f^m_n\in\dF(F_m,F_n)$ is a sequence such that 
	$$\Bbb F=\varprojlim\left(\cdots\to F_{n+1}\xrightarrow{p\left(f^{n+1}_{n}\right)} F_n\xrightarrow{p\left(f^{n}_{n-1}\right)} \cdots\xrightarrow{p\left(f^1_0\right)} F_0\right)$$
	is a $W_P$-prespace. We want to show that $\langle F_n\mid <\omega\rangle$ is a \Fraisse sequence for $\dF$. In other words given $A\in\dF$ and an epimorphism $g\in\dF(A,F_n)$, we want to find $m\geq n$ and an epimorphism $h\in\dF(F_m,A)$ such that $g\circ h=f^m_n$. As in Example \ref{remark: Immersion in limit of sequence} we obtain an immersion $h_n\colon F_n\to\Bbb F$. Moreover we also have an immersion $j\colon F_n\to A$ as in Example \ref{remark: arrows produce immersions}. Using the fact that $\Bbb F$ is a $W_P$-prespace we can find an immersion $k\colon A\to\Bbb F$ such that $k\circ j=h_n$. We construct $k$ inductively on elements of $A$. There is only one choice for the value of $k$ on $j(F_n)$ so that $k\circ j=h_n$. Enumerate $A\setminus j(F_n)$ as $\{a_1,\ldots,a_r\}$. We define $k(a_1)$ to be any point of $\Bbb F$ with $\ord(a_1)=\ord(k(a_1))$ so that $k$ is an isomorphism of trees with the betweenness relation from $j(F_n)\cup\{a_1\}$ to $h_n(A)\cup\{k(a_1)\}$. There exists such a point $k(a_1)\in\Bbb F$ because $\Bbb F$ is a $W_P$-prespace. Proceeding inductively in a similar manner, we define $k(a_i)$ for all $i\leq r$. Now we can find $m\geq n$ such that, calling $f^\infty_m\in\dF(\Bbb F,F_m)$ the canonical projection, all vertices in $k(A)$ have distinct images through $f^\infty_m$. We check that $i=p(f^\infty_m)\circ k$ is an immersion of $A$ in $F_m$. If $a,b,c\in A$ are such that $b\in[a,c]$, then $k(b)\in[k(a),k(c)]$ because $k$ is an immersion, and $i(b)\in[i(a),i(c)]$ because $p(f^\infty_m)$ is monotone so it must map arcs to arcs. Conversely suppose for a contradiction that $i(b)\in[i(a),i(c)]$ but $b\not\in[a,c]$. Let $d\in[a,c]$ be the only point of $[a,c]$ which belong to the shortest arc from $b$ to $[a,c]$. If $d=a$, then $a\in[b,c]$, which, arguing as in the previous implication, implies $k(a)\in[k(b),k(c)]$ and so $i(a)\in[i(b),i(c)]$, a contradiction. If $d=c$ we reach a contradiction in the same manner, which only leaves the case $d\in(a,c)$. But in the latter case we have $k(d)\in[k(a),k(c)]$ and $k(b)$ is in a component of $\widehat{k(d)}$ that contains neither $k(a)$ nor $k(c)$, since $k$ is an immersion. By assumption $p(f^\infty_m)(a),p(f^\infty_m)(b),p(f^\infty_m)(c)$ and $p(f^\infty_m)(d)$ are pairwise distinct, so we must have that $i(d)\in(i(a),i(c))$, since $p(f^\infty_m)$ maps arcs to arcs, but we must also have that $i(b)$ is in a component of $\widehat{i(d)}$ containing neither $i(a)$ nor $i(c)$, a contradiction. 

 We can now define a weakly coherent map from $F_m$ to $A$, we only need to be careful and make sure that it makes the appropriate diagram commute. Explicitly we define $h\in\dF(F_m,A)$ as follows. Let $x\in F_m$. If $x=i(a)$ for some $a\in A$, then $p(h)(x)=a$. If $x\in \widehat{a}_{i(A)}$ (as defined in Definition \ref{def: center closed stuff}) for some $a\in A$, then $p(h)(x)=a$. If $x\in C_{i(a),i(a')}$ for some $a,a'\in A$ with $(a,a')\in E(A)$, then we check whether $p(g)(a)=f^m_n(x)$ or $p(g)(a')=f^m_n(x)$ and define $p(h)(x)$ accordingly. Since the map we just defined is monotone, it must also be weakly coherent by Remark \ref{remark: from immersion to coherence}. For the embedding part of $h$ consider $a\in\End(A)$. If $a\in e(g)(\End(F_n))$ set $e(h)(a)=e(f^m_n)(b)$, where $b\in\End(F_n)$ is such that $e(g)(b)=a$. If instead $a$ is not in the image of $e(g)$, then $e(h)(a)$ can be any endpoint in $p(h)^{-1}(a)$.
 Note that $p(h)\circ e(h)(a)=a$ for any $a\in \End(A)$ and  $g\circ h= f^m_n$.
\end{proof}

\noindent We conclude by proving the following theorem, from which Theorem \ref{thm: CDH} follows immediately by the uniqueness of \Fraisse limits.

\begin{thm}\label{thm: limit of dF with countable dense set of endpoints}
    Let $P\subseteq\{3,4,\ldots,\omega\}$ and let $Q\subseteq \End(W_P)$ be a countable dense set of endpoints. Then $\dF$ has a \Fraisse sequence $\overline{u}=\langle F_i\mid i<\omega\rangle$ with morphisms $f^m_n\in\dF(F_m,F_n)$ such that if $(\Bbb U,D)$ is its \Fraisse limit, then $\pi(D)=Q$, where $\pi\colon \Bbb U\to |\Bbb U|\cong W_P$ is the topological realization.
\end{thm}
\begin{proof}
Enumerate $Q$ without repetitions as  $\{q_1,q_2,q_3,\ldots\}$. Let $A_0=\{q_1,q_2,q_3,c\}$ where $c=C(q_1,q_2,q_3)$ and, for all $0<n<\omega$, let $A_n$ be the center closure of $\{q_1,\ldots,q_{k_n}\}$, where $(k_n)_{n<\omega}$ is an increasing sequence chosen so that $|A_{n+1}\cap(a,a')|\geq 2$, whenever $a,a'\in A_n$ are such that $[a,a']\cap A_n=\{a,a'\}$, which is possible since $Q$ is dense. We now build a finite graph $F_n$ with a vertex $v_a$ for every $a\in A_n$, and an edge $\langle v_a,v_{a'}\rangle$ iff $[a,a']\cap A_n=\{a,a'\}$. We label every $v_c\in F_{n}$ that comes from a ramification point $c\in A_{n}$ with the order of $c$ in $W_P$. It only remains to define morphisms $f^{n+1}_n\in\dF(F_{n+1},F_n)$. Since we have $A_n\subseteq A_{n+1}$, we also get an immersion $i\colon F_n\to F_{n+1}$ that maps $v_a\in F_n$ to $v_a\in F_{n+1}$, for every $a\in A_n$. The embedding part of $f^{n+1}_n$ is exactly $\restr{i}{\End(F_n)}$. For the projection part for all pairs $a,a'\in A_n$ with $[a,a']\cap A_n=\{a,a'\}$, we partition $C_{v_{a},v_{a'}}\cup\{v_{a},v_{a'}\}\subseteq F_{n+1}$ into two sets $H,H'$ such that 
\begin{enumerate}
\item $v_a\in H$, $v_{a'}\in H'$,
\item $H,H'$ are connected,
\item $|H|,|H'|\geq 2$.
\end{enumerate} 
Note that the last condition is easily satisfied since we assumed that $|(a,a')\cap A_{n+1}|\geq 2$. We then define $p(f^{n+1}_n)(H)=v_a$ and $p(f^{n+1}_n)(H')=v_{a'}$. By construction $p(f^{n+1}_n)$ is monotone and satisfies the hypothesis of Remark \ref{remark: from immersion to coherence}, so it must be weakly coherent.

Thanks to Theorem \ref{thm: transitive edges} and the condition $|(a,a')\cap A_{n+1}|\geq 2$ the limit of this sequence is a prespace. Since $Q$ is dense we have that $\bigcup_n A_n=\mathrm{Br}(W_P)\cup Q$. Using this fact it is easy to check that the limit of this sequence is a $W_P$-prespace, so that the sequence is \Fraisse by Theorem \ref{thm: sequence is Fraisse if limit is WP-prespace}. 
Calling $(\Bbb U,D)$ its limit, we now only need to identify $(|\Bbb U|,\pi(D))$ with $(W_P,Q)$. To any $a\in A_{n+1}\setminus A_n$ we have associated the point $$r_a=(p(f^{n+1}_0)(v_a),\ldots,p(f^{n+1}_{n-1})(v_a),p(f^{n+1}_n)(v_a),v_a,\ldots,v_a,v_a,\ldots)$$ of $\Bbb U$, which is a ramification point by Theorem \ref{thmw: coherent implies splitting} if $a$ is a ramification point of $A_{n+1}$, and which is an endpoint if $a$ is an endpoint of $A_{n+1}$. We define a continuous function $h\colon \mathrm{Br}(W_P)\cup Q\to|\Bbb U|$ by setting $h(a)=\pi(r_a)$ for every $a\in\bigcup_n A_n$. To verify the continuity of $h$ note that if $x_n\to x$ is a convergent sequence of points in the domain of $h$ whose limit is also in the domain of $h$, then for any $m\in\mathbb N$ we can find $N\in\mathbb N$ big enough so that $A_m$ doesn't separate $x_k$ and $x$, for any $k\geq N$. In other words $r_{x_k}$ and $r_x$ agree in at least the first $m$ coordinates, which implies that $r_{x_k}\to r_x$, since $m$ was arbitrary. This shows that the function $a\mapsto r_a$ is continuous, from which the continuity of $h$ follows immediately, since $h$ is obtained by composition of that function with $\pi$, which is continuous. Note that the image of $h$ is dense in $|\Bbb U|$ (in particular, any ramification point of $|\Bbb U|$ is in the image), $h$ is injective, and $h$ preserves the betweenness relation. We will now extend $h$ to a continuous function $\tilde h\colon W_P\to|\Bbb U|$ and then verify that $\tilde h$ is injective and surjective. For $x\in W_P$, let $$\osc(h,x)=\inf\{\diam h(U)\mid U\text{ is an open neighbourhood of }x\},$$ where $h(U)$ means $h(U\cap(\mathrm{Br}(W_P)\cup Q))$. By the Kuratowski's extension theorem (see Theorem 3.8 and its proof in \cite{Kechris}), in order to extend the continuous function $h$ to a continuous function $\tilde h\colon W_P\to|\Bbb U|$, it suffices to show that $\osc(h,x)=0$ for every $x\in\End(W_P)\setminus Q$ or when $x$ is a regular point.

\begin{itemize}
    \item In the first case let $(x_i)_{i<\omega}$ be a sequence in $\mathrm{Br}(W_P)$ with $x_{i+1}\in[x_i,x]$ for every $i$, $x_i\to x$ and, for every $i<\omega$, let $C_{x_i}(x)$ be the component of $W_P\setminus\{x_i\}$ containing $x$. Then $\{x\}=\bigcap_{i<\omega}C_{x_i}(x)$ and we have $\bigcap_{i<\omega}h(C_{x_i}(x))=\varnothing$. Indeed if $y\in\mathrm{Br}(W_P)\cup Q$ and $w$ is the unique point on $[y,x_0]\cap[y,x]\cap[x_0,x]$, then there is $i$ big enough so that $w\not\in C_{x_i}(x)$, so that $y\not\in C_{x_i}(x)$ and, since $h$ preserves the betweenness relation, $h(y)\not\in C_{x_i}(x)$. Since $C_{x_{i+1}}(x)\subseteq C_{x_i}(x)$ for every $i$ (because $x_{i+1}\in [x_i,x]$), this shows that $\bigcap_{i<\omega}h(C_{x_i}(x))=\varnothing$. Therefore  by compactness of $W_P$ we have $\diam(h(C_{x_i}(x) ))\to 0$ and hence 
    $\osc(h,x)=0$.
    
    \item If $\ord(x)=2$ let $(y_i)_{i<\omega}$ and $(z_i)_{i<\omega}$ be two sequences in $\mathrm{Br}(W_P)$ with $y_{i+1}\in[y_i,x]$, $y_i\to x$, $z_{i+1}\in[x,z_i]$, $z_i\to x$ and $x\in [y_i,z_i]$ for every $i<\omega$. Letting $C_{y_i,z_i}(x)$ denote the component of $X\setminus\{y_i,z_i\}$ containing $x$, we can argue as in the previous case to obtain that $\{x\}=\bigcap_{i<\omega}C_{y_i,z_i}(x)$ and that $\bigcap_{i<\omega}h(C_{y_i,z_i}(x))=\varnothing$. Similarly to the previous case, since $C_{y_{i+1},z_{i+1}}(x)\subseteq C_{y_i,z_i}(x)$ for every $i$, this shows that $\osc(h,x)=0$.
\end{itemize}  

As discussed above we can now extend $h$ to a continuous function $\tilde h\colon W_P\to |\Bbb U|$. 
Note that $\tilde h$ preserves the betweenness relation. Indeed, the betweenness relation $\{(x,y,z)\in W_P^3 \mid x\in[y,z]\}$ is closed in $W_P^3$. Therefore if $x_n\to x$, $y_n\to y$, $z_n\to z$, $x_n\in [y_n,z_n]$, $x\in[y,z]$,
$x_n, y_n, z_n\in \mathrm{Br}(W_P)$, it holds $\tilde h(x_n)\in [\tilde h(y_n), \tilde h(z_n)]$, and by continuity of $\tilde h$, we get $\tilde h(x)\in [\tilde h(y), \tilde h(z)]$.

Since $\tilde h(W_P)\supseteq \overline{h(\mathrm{Br}(W_P)\cup Q)}=|\Bbb U|$, $\tilde h$ must be surjective. 

To show injectivity of $\tilde h$ take $x\neq y\in W_P$ such that none  of the $x,y$ is  a ramification point. Since ramification points are arcwise dense in $W_P$, pick $t\in \mathrm{Br}(W_P)$ with $t\in [x,y]$. Then $\tilde h(t)\in [\tilde h(x), \tilde h(y)]$ and since none of the $\tilde h(x), \tilde h(y)$ is a ramification point, we obtain $\tilde h(t)\neq \tilde h(x), \tilde h(y)$. Therefore $\tilde h(x)\neq \tilde h(y)$.

Moreover we have $\tilde{h}(Q)=\pi(D)$ by construction, since $\tilde h(Q)=h(Q)$, which concludes the proof.
\end{proof}

\section*{Acknowledgement}
We would like to thank the referee for carefully reading the article and for a number of suggestions on the presentation of the results.

\normalsize




\normalsize
\baselineskip=17pt


\printbibliography

\end{document}